\documentclass[leqno]{amsart}
\usepackage{amsmath}
\usepackage{amsthm}
\usepackage{amssymb}
\usepackage{latexsym}
\usepackage{enumerate}

\newtheorem{thm}{Theorem}[section]
\newtheorem{prop}{Proposition}[section]
\newtheorem{lem}{Lemma}[section]
\newtheorem{cor}{Corollary}[section]
\newtheorem{rem}{Remark}[section]

\title{Perturbation of a warped product metric of an end and the growth property of solutions to eigenvalue equation${}^{*}$}

\author{Hironori Kumura${}^{**}$}
\email{smhkumu@ipc.shizuoka.ac.jp} 
\address{Department of Mathematics, Shizuoka University, 
Shizuoka 422-8529, Japan} 

\thanks{${}^{*}$\
}

\thanks{${}^{**}$\ 
supported in part by the Grants-in-Aid for Scientific Research (C), 
Japan Society for the Promotion of Science, No.~21540215.} 

\date{}

\begin{document}

\maketitle

\begin{abstract} 
In this paper, we shall study Riemannian metrics around a warped product one, and derive growth estimates of solutions to the eigenvalue equation, from which the absence of eigenvalues will follows. 
\end{abstract}
%%
%%%%%%    INTRODUCTION    %%%%%%
%%
\section{Introduction}

The Laplace-Beltrami operator $\Delta$ on a noncompact complete Riemannian manifold $(M,g)$ is self-adjoit on $L^2(M)$ and it has been studied by several authors from various points of view. 
Especially, the problem of the absence of eigenvalues was discussed in $[6-11]$, $[14]$, $[15]$, $[19-21]$, $[23]$, $[25]$, and so on. 
In particular, using Carleman inequalities, Donnelly \cite{D4} proved the non-existence of positive eigenvalues under the assumption that $(M,g)$ admits an exhaustion function $b$ satisfying $| \nabla db - b^{-1} (g - db \otimes db)| \le c_1 r^{-1-\varepsilon}$, $\left| |\nabla b|-1 \right| \le c_2 r^{-\varepsilon}$, $c_3 r \le b \le c_4 r$ in the complement of a compact subset of $M$; his result can be seen to be applied for ALE spaces by the work of Bando-Kasue-Nakajima \cite{B-K-N}. 
Moreover, Donnelly \cite{D5} used the Mourre theory and proved the non-existence of positive eigenvalues under the assumption that $(M,g)$ admits an exhaustion function $b$ satisfying $| \nabla db - b^{-1} (g - db \otimes db)| \le c_1 r^{-1-\varepsilon(r)}$, $1 - \varepsilon(r) \le |\nabla b| \le 1 + \varepsilon(r)$, $c_1 r \le b \le c_2 r$, $| d\Delta b^2 | \le \varepsilon(r)$ in the complement of a compact subset of $M$, where $\varepsilon(r)$ stands for a positive valued function satisfying $\lim_{r \to \infty} \varepsilon(r)=0$; his result can be seen to be applicable for manifolds with non-negative Ricci curvature, Euclidean volume growth, and quadratic curvature decay by the works due to Cheeger-Colding \cite{C-C} and Colding-Minicozzi \cite{C-M}. 
On the other hand, the author \cite{K4} proved the non-existence of eigenvalues under the assumption that $(M,g)$ has an end $E$ with radial coordinates and $-b(b-1)r^{-2} \le {\rm radial~curvature} \le a(a-1)r^{-2}$ on $E$, where $a$ and $b$ are constants satisfying $\max\{ \frac{n-1}{n+1}b, b-\frac{2}{n-1} \} < a \le b$ and $\dim M=n$. 
The purpose of this paper is to investigate these problems further; in this paper, we claim that the lower bound of radial curvatures can be replace by $-o(r^{-1})$ if the metric is close to warped product metric. 
Our method is a modification of solutions of Kato \cite{Kato}, Eidus \cite{E}, Roze \cite{R}, and Mochizuki \cite{M}, to the analogous problem for the Schr\"odinger equation on Euclidian space; in particular, we shall use integration-by-parts arguments. 

We shall introduce some terminology and notations to state our results. 
Let $(M,g)$ be an $n$-dimensional noncompact complete Riemannian manifold and $U$ an open subset of $M$. 
We shall say that $E:=M-U$ is an {\it end with radial coordinates} if and only if the boundary $\partial U$ is compact, connected, and smooth, and the outward normal exponential map $\exp_{\partial U}^{\perp} : N^{+}(\partial U) \to M - \overline{U}$ induces a diffeomorphism, where $N^{+}(\partial U) = \{ v\in T(\partial U) \mid v {\rm ~is~outward~normal~to~}\partial U \}$. 
Note that $U$ is not necessarily relatively compact. 
Let $r$ denote the distance function from $\partial U= \partial E$ defined on the end $E$. 
We shall say that a $2$-plane $\pi \subset T_xM$ $(x\in E)$ is {\it radial} if $\pi$ contains $\nabla r$, and, by the {\it radial curvature}, we mean the restriction of the sectional curvature to all the radial planes. 
In the sequel, the following notations will be used:
\begin{align*}
 & B(s,t) = \{ x\in E \mid s < r(x) < t \} \quad 
            \mathrm{for}~~0 \le s < t ;\\
 & B(s,\infty) = \{ x\in E \mid s<r(x) \} \quad 
            \mathrm{for}~~0 \le s <\infty;\\
 & S(t) = \{ x\in E \mid r(x)=t \} \quad 
            \mathrm{for}~~0 \le t < \infty ;\\
 & \sigma(-\Delta) =  \,{\rm the~spectrum~of}~-\Delta;\\
 & \sigma_{\rm pp}(-\Delta) =  \,{\rm the~set~of~all~eigenvalues~of}~
   - \Delta;\\
 & K_{{\rm rad.}} =  \,{\rm the~radial~curvature~on~}E;\\
 & \widetilde{g} = g - dr \otimes dr \quad {\rm on}~E.
\end{align*}
Moreover, we denote the Riemannian measure of $(M,g)$ by $dv_g$, and the induced measures from $dv_g$ on each $S(t)~( t > 0 )$ simply by $dA$. 

We shall state our theorems; Theorem $1.1$ and $1.2$ are concerned with warped product case:
%%
%%
%%%%%%%%%%%  Theorem 1.1   %%%%%%%%%%%
%%
\begin{thm}
Let $(M,g)$ be an $n$-dimensional noncompact complete Riemannian manifold and $E$ be an end of $M$ with radial coordinates. 
Assume that the metric $g|_E$ restricted on $E$ has the following form:
\begin{align*}
  g|_E = dr^2 + f(r)^2 g_{\partial E},
\end{align*}
where $r = {\rm dist}\,(\partial E,*)$ on $E$, and $g_{\partial E}$ stands for the induced metric on $\partial E$. 
We denote $A_f(r) = \frac{f'(r)}{f(r)}$ and assume that there exist constants $r_0 > 0$ and $a > 0$ such that 
\begin{align*}
 & A_f(r_0)
  \begin{cases}
    \ge \displaystyle \frac{a}{r_0} & \text{if} \quad 0 < a < 1,\vspace{1mm} \\
    > 0   & \text{if} \quad a \ge 1;
  \end{cases}\\
 & - \frac{\varepsilon(r)}{r} 
   \le K_{{\rm rad.}} 
   \le \frac{a(1-a)}{r^2} \qquad \mathrm{on}~~B(r_0,\infty),
\end{align*}
where $\varepsilon (t)$ is a positive-valued function of $t \in [r_0,\infty)$ satisfying $ \lim_{t \to \infty} \varepsilon (t) = 0 $. 
Then, $\sigma (-\Delta) = \left[0,\infty \right)$ and $-\Delta $ has no eigenvalue. 
\end{thm}
Note that the shape operator $A_f(r)$ of level hypersurfaces $S(r)$ of $r$ is positive definite for any $t$ in $[r_0,\infty)$, which implies that $E$ is an expanding end. 
If the end $E$ of an rotationally symmetric Riemannian manifold shrinks at infinity, infinitely many eigenvalues appear; also, if the end $E$ is asymptotically cylinder, that is, $\lim_{r \to \infty} f(r) >0$ and $f$ is not constant near infinity, then infinitely many eigenvalues will appear. 

Note also that we can take $f(r)=r^p~(p>0)$ and $f(r)=\exp r^p~(0<p<\frac{1}{2})$ in Theorem $1.1$ and Theorem $1.3$ which is stated later, and hence, fairly wider classes of model manifolds than Euclidean space case (that is the case that $f(r)=r$) are applicable. 

The decay order $ - \frac{\varepsilon(r)}{r} $ of the lower bound of radial curvatures in Theorem $1.1$ seems to be sharp; it is interesting to compare $ - \frac{\varepsilon(r)}{r} $ and the radial curvature $2k \frac{\cos 2r}{r}$ in the following:
%%
%%%%%%%    Theorem 1.2    %%%%%%%
%%
\begin{thm}
There exists a rotationally symmetric manifold 
$(M,g) = \bigl( {\bf R}^n, dr^2 + f^2(r) g_{S^{n-1}(1)} \bigr)$ with the following three properties$:$
\begin{enumerate}[$(i)$]
  \item $\nabla dr 
= \displaystyle\left( 2k \frac{\cos 2r}{r} + O(r^{-1-\alpha}) \right) \widetilde{g}$ as $r\to \infty$, and hence, $\sigma (-\Delta) = \left[0,\infty \right)$ $;$   
  \item $\sigma_{\rm pp} ( - \Delta ) = \{ 1 \}$$;$
  \item $\displaystyle K_{{\rm rad.}} = -2k \frac{\cos 2r}{r} + O(r^{-1-\alpha}) $ as $r\to \infty$.
\end{enumerate}
Here, $k$ and $\alpha$ are constants satisfying $(n-1)|k|>2$ and $\frac{1}{2}<\alpha<1$. 
\end{thm}
By using the comparison theorem in Riemannian geometry (Kasue \cite{Kasue} is a good reference for the comparison theorem in Riemannian geometry), Theorem $1.1$ follows from the following Theorem $1.3$ which we shall now explain. 

Let us consider `$1/r$-perturbation' of shape operators of level hypersurfaces of $r$ as in the inequality $(1)$ below: let $(M,g)$ be an $n$-dimensional noncompact complete Riemannian manifold and $E$ an end of $M$ with radial coordinates, and denote $r = {\rm dist}(\partial E,*)$ on $E$. 
Assume that there exists $r_0 > 0$ such that 
\begin{align}
  \left\{ \frac{f'(r)}{f(r)} - \frac{a}{r} \right\}  \,\widetilde{g} 
  \le \nabla dr 
  \le \left\{ \frac{f'(r)}{f(r)} + \frac{b}{r} \right\}  \,\widetilde{g} 
  \qquad {\rm on}~~B(r_0,\infty), 
\end{align}
where $f(t)$ is a positive-valued $C^{\infty}$ function of $t\in [r_0,\infty)$, and $a$ and $b$ are positive constants. 
In the sequel, we shall often use the following notations for simplicity:
\begin{align*}
 & \widehat{a} = ( n - 1 ) a \qquad {\rm for~a~constant}~a>0;\\
 & A(r) = \frac{f'(r)}{f(r)};~K(r) = - \frac{f''(r)}{f(r)};\\
 & \widehat{A}(r) = (n-1) A(r);~\widehat{K}(r) = (n-1)K(r).
\end{align*} 
For the sake of convenience, we shall list the assumptions used in this paper:
\begin{align}
 & \frac{A_0}{r} \le A(r) \le \frac{ B_0 }{ \sqrt{r} } 
   \qquad {\rm for}~~r \ge r_0 ; \\
 & 2( A_0 - a ) > \widehat{a} + \widehat{b} ; \\
 & \left| r \widehat{A}\,'(r) \right| 
   = \left| - r \left( \widehat{A}(r)A(r) + \widehat{K} (r) \right) \right| 
   \le \widehat{K}_3 \qquad {\rm for}~~r \ge r_0 ; \\
 & {\rm Ric}_g (\nabla r, \nabla r) 
   \ge  - \frac{ \widehat{b}_1 }{r} \qquad {\rm on}~~B(r_0,\infty), 
\end{align}
where $A_0$, $B_0$, $b_1$, and $K_3$ are positive constants. \\
Moreover, for $\gamma > \frac{ \widehat{a} + \widehat{b} }{2}$, we denote 
\begin{align*}
  & \lambda_1(\gamma,a,b,A_0,B_0,b_1,K_3,b_1)\\
= & \max \left\{ 
  \frac{(\widehat{K}_3)^2}{( y_1 -\widehat{a} )(2\gamma - \widehat{b}- y_1 )}, 
  \left( \frac{ 4\widehat{b}_1 + (\widehat{B}_0)^2 }
  { 8\left( 2(A_0 - a) - \widehat{a} - \widehat{b} \right) } 
  \right)^2 \right\},
\end{align*}
where $y_1= \min \left\{ \frac{2\gamma + \widehat{a} - \widehat{b}}{2}, 2(A_0-a) - \widehat{b} \right\}$. 
Note that $\lambda_1(\gamma,a,b,A_0,B_0,K_3,b_1)$ converges to zero as $K_3$, $b_1$, and $B_0$ tend to zero. \\
Note also that $(1)$ and $(2)$ imply that $\lim_{r\to \infty}\Delta r=0$, and hence, $\sigma(-\Delta )=[0,\infty)$. 
%%
%%%%%%%%%%%   THEOREM 1.3   %%%%%%%%%%%
%%
\begin{thm}
Let $(M,g)$ be an $n$-dimensional noncompact complete Riemannian manifold and $E$ an end of $M$ with radial coordinates. 
We denote $r = {\rm dist}(\partial E,*)$ on $E$. 
Assume that there exist a positive-valued $C^{\infty}$ function $f(t)$ of $t \in [r_0,\infty)$ and positive constants $r_0$, $a$, $b$, $A_0$, $B_0$, $b_1$, and $K_3$ such that $(1)$, $(2)$, $(3)$, $(4)$, and $(5)$ hold. 
Let $\lambda > 0$ be a constant and $u$ a solution to 
\begin{align*}
   \Delta u + \lambda u = 0 \qquad {\rm on}~~B(r_0,\infty).
\end{align*}
Let 
\begin{align*}
  \gamma > \frac{ \widehat{a} + \widehat{b} }{2}  
\end{align*} 
be a constant and assume that $u$ satisfies the condition$:$
\begin{align*}
  \liminf_{t\to \infty} ~ t^{\gamma} \int_{S(t)}
  \left\{ \left( \frac{\partial u}{\partial r} \right)^2 + |u|^2 
  \right\} \,dA=0.
\end{align*}
Moreover, we assume that 
\begin{align*}
  \lambda > \lambda_1( \gamma, a, b, A_0, B_0, K_3 ).
\end{align*}
Then, $u \equiv 0$ on $B(r_0,\infty)$. 

In particular, if $ 2 > \widehat{a} + \widehat{b} $, then $\sigma_{{\rm pp}}(-\Delta) \cap \bigl( \lambda_1(1, a, b, A_0, B_0, K_3), \infty \bigr) =\emptyset$. \end{thm}
\vspace{3mm}

\noindent
{\bf Acknowledgement}~~The author should like to express his gratitude to a referee for his many appropriate comments and careful readings. 
%%
%%
%%%%%%   Section 2     %%%%%%
%%
\section{Proof of Theorem $1.2$}

In this section, we shall prove Theorem $1.2$. 

First, recall the following theorem essentially due to Atkinson \cite{A}; for the proof of Lemma $2.1$, see Arai-Uchiyama \cite{A-U} and references there. 
%%
%%%%%%%   Lemma 2.1   %%%%%%%
%%
\begin{lem}
Let $\lambda > 0 $, $\varepsilon > 0 $, and $k \in {\bf R}$  be constants and 
$q(x) \in C^0[0,\infty)$ a real-valued function. 
Assume that 
\begin{align*}
  q(x) = - k \, \frac{\sin 2x}{x} + O \left( x^{-1-\varepsilon} \right)
  \qquad ( x \to \infty)
\end{align*}
and consider the eigenvalue equation
\begin{align}
  \left( - \frac{d^2}{dx^2} + q(x) \right) w (x) = \lambda w(x) 
  \qquad {\rm on}~~[0,\infty).
\end{align}
Then, the following properties $(a)$ and $(b)$ are equivalent$\,:$
\begin{enumerate}[$(a)$]
  \item  The equation $(6)$ has a nontrivial solution $w\in L^2[0,\infty)$$;$
  \item  $|k|>2$ and $\lambda = 1 $.
\end{enumerate}
\end{lem}
\vspace{5mm}

Firstly, except for a positive multiplier, we shall construct a desired function $f$ on a neighborhood of infinity as follows: for $r\ge 1$, let $\alpha$ and $k$ be constants satisfying $ \frac{1}{2} < \alpha < 1$ and $|k| > \frac{n-1}{2}$, respectively. 
\begin{align}
  f_1(r) = \exp \left\{ \int_1^r \left( \frac{1}{t^{\alpha}} 
  + k \frac{\sin 2t}{t} \right) dt \right\},
\end{align}
and consider the Riemannian manifold with boundary:
\begin{align*}
  (N, g_N) = \Bigl( [1,\infty ) \times S^{n-1}(1) 
  , dr^2 + f_1^{\,2}(r)g_{S^{n-1}(1)} \Bigr),
\end{align*}
Then,
\begin{align}
 & A_{f_1}(r) = \frac{f_1'(r)}{f_1(r)} = 2k \frac{\cos 2r}{r} 
   + O\left( \frac{1}{r^{1 + \alpha}} \right); \\
 & K_{f_1}(r) = - \frac{f_1''(r)}{f_1(r)} 
   = - 2k \frac{\cos 2r}{r} + O\left( \frac{1}{r^{ 2 \alpha }} \right);\\
 & \frac{1}{f_1(r)} = O\left( \frac{1}{r^{ m }} \right) \quad {\rm for~any~positive~integer}~m. \nonumber
\end{align}
Thus,
\begin{align}
  q_0(r) & = \frac{(n-1)(n-3)}{4} A_h^2(x) - \frac{(n-1)}{2} K_h(x)\nonumber \\
  & = k(n-1) \frac{\cos 2r}{r} + O(r^{-2\alpha}). \nonumber
\end{align}
Since $ \frac{1}{2} < \alpha < 1$ and $|k|(n-1)>2$, Lemma $2.1$ implies that there exists a nontrivial solution $w (x) \in L^2([1,\infty),dx)$ to the equation
\begin{align*}
  \left( -\frac{d^2}{dx^2} + q_0(x) \right) w (x) = w (x).
\end{align*}
Note that $w$ {\it oscillates around $0$, changes the sign}, and $\lim_{x \to \infty} |w|(x) = 0 $ (see \cite{A} and \cite{A-U}). 
Using this function $w$, we define a function $h$ by
\begin{align}
   h := f_1^{-\frac{n-1}{2}}w .
\end{align}
Then, a direct computation shows that the function $h \bigl( r(p) \bigr)$ $( p \in N) $ satisfies the eigenvalue equation on $(N,g_N)$:
\begin{align*}
  - \Delta_{g_N} \bigl( h(r)  \bigr) 
  = -\left\{ \frac{\partial ^2}{\partial r^2}
  + (n-1) A_{f_1}(r) \frac{\partial }{\partial r} \right\} h(r)
  =  h(r)
\end{align*}
and $h(r)\in L^2(N,dv_{g_N})$. 
Note that $dv_{g_N} = f_1^{\,n-1}(r)\,dr dv_{g_0}$, where $dv_{g_0}$ is the standard measure on the unit sphere $( S^{n-1}(1), g_0)$. 

Secondly, we shall construct a neighborhood of the origin of the desired manifold. 
Let $B_{{\bf R}^n}(0,r)$ be an open ball of radius $r$ and centered at the origin $0$ in the Euclidean space $({\bf R}^{n},g_{{\rm stand}})$ and denote by $\lambda _1 \bigl( B_{{\bf R}^n} (0,r) \bigr)$ the first Dirichlet eigenvalue of $B_{{\bf R}^n}(0,r)$. 
Since $\lim_{r\to +0} \lambda _1 \bigl( B_{{\bf R}^n} (0,r) \bigr) = \infty $ and $\lim_{r\to \infty } \lambda _1 \bigl( B_{{\bf R}^n} (0,r) \bigr) = 0$, there exists $ r_1 > 0 $ such that $ \lambda _1 \bigl( B_{{\bf R}^n} (0,r_1) \bigr) =  1$. 
Let $\widetilde{\varphi}_1$ be its associated positive-valued first eigenfunction. 
Since $\widetilde{\varphi}_1$ is a radial function, it can be written as $\widetilde{\varphi}_1 = H(r)$, where $r$ stands for the Euclidean distance to $0$.  
We note that $ H' < 0 $ on $(0,r_1]$. 

Thirdly, we shall connect two parts mentioned above; in view of $(7)$ and $(10)$, the function $h(t)$ also oscillates around $0$, changes the sign, and converges to $0$ as $t \to \infty$. 
Hence, there exist a constant $ r_2 > \max\{ r_1, 1\} $ such that $h(r_2) < 0 $ and $h'(r_2) < 0 $. 
Therefore, we can connect two functions, $H$ on $[0,r_1]$ and $h|_{[r_2,\infty)}$, by some function $\psi \in C^{\infty}[0,\infty)$ satisfying 
\begin{align}
  \psi (x)  = 
 \begin{cases}
   H(x) & \qquad \mathrm{if}~~x \in [0 , r_1], \\
   h(x) & \qquad \mathrm{if}~~x \in [r_2, \infty) ; 
 \end{cases}
\end{align}
and 
\begin{align}
  \psi ' (x) < 0 \qquad \mathrm{if}~~x \in [r_1 , r_2].
\end{align} 
Now, let us construct a function $f$ so that $\psi$ is an eigenfunction with eigenvalue $1$ on $\bigl( {\bf R}^n, dr^2 + f^2(r) g_{S^{n-1}(1)} \bigr)$, 
that is,
\begin{align}
  \psi''(r) + (n-1) \frac{f'(r)}{f(r)}\psi'(r) = - \psi(r).
\end{align}
By $(11)$ and $(12)$, we have $\psi ' < 0$ on $(0, r_2]$. 
Hence, we can solve the differential equation $(13)$ on the interval $[0,r_2]$ with the condition $f(r_1)=r_1$: 
\begin{align}
  f(r) = r_1 \exp \left\{ - \int_{r_1}^r 
      \frac{\psi(s)+\psi''(s)}{(n-1)\psi'(s)} \,ds \right\}
     \qquad {\rm for}~~r \in [0,r_2].
\end{align}
Since $\psi |_{[0,r_1]}= H$ and $\widetilde{\varphi}_1 = H(r)$ on $B_{{\bf R}^n} (0,r)$, we see that $f(t) = t$ on $[0,r_1]$, and hence, $(B_{{\bf R}^n}(0,r_1), dr^2 + f^2(r)g_{S^{n-1}(1)})$ is a flat disk in $({\bf R}^{n},g_{{\rm stand}})$ with radius $r_1$; next, using this function $f$ on $[0,r_2]$, let us set 
\begin{align}
   f(r) = \frac{f(r_2) }{ f_1(r_2) } f_1 (r)
   \qquad {\rm for}~~r \in [r_2,\infty).
\end{align}
Then, this positive-valued function $f$ on $(0,\infty)$, defined by $(14)$ and $(15)$, satisfies the equation $(13)$, and hence, we see that $\psi$ is an eigenfunction with eigenvalue $1$ on the manifold $(M,g):=\bigl( {\bf R}^n, dr^2 + f^2(r) g_{S^{n-1}(1)} \bigr)$; 
\begin{align}
   1 \in \sigma_{{\rm p}}(-\Delta).
\end{align}
Moreover, from $(15)$, we have for $r \ge r_2$
\begin{align}
  & \nabla dr = A_{f_1}(r) \{ g - dr \otimes dr \} \\
  & K_{{\rm rad.}} = K_{f_1}(r) ,
\end{align}
and hence, Theorem $1.2$ $(i)$ follows from $(8)$ and $(15)$; Theorem $1.2$ $(iii)$ follows from $(9)$ and $(18)$. 

In order to prove that $ - \Delta$ has no eigenvalue on the interval $\left( 0 ,\infty \right) $ except for the special number $1$, we shall use the separation of variables: 
${\bf R}^n - \{ 0 \}$ is diffeomorphic to $(0,\infty) \times S^{n-1}(1)$ and 
we denote the eigenvalues of the Laplacian on the standard unit sphere $S^{n-1}(1)$ by
$$
  0 = \mu_0 < \mu_1 \le \mu_2 \le \cdots
$$ 
with repetitions according to multiplicity; 
then, $-\Delta$ on $(M,g) = ({\bf R}^n,dr^2 + f^2(r) g_{S^{n-1}(1)})$ is unitarily equivalent to the infinite sum of the operators $-\widetilde{L}_j$ on $L^2\bigl( (0,\infty),dx \bigr)$:
\begin{align*}
  & - \widetilde{L}_j = - \frac{d^2}{dx^2} + \widetilde{q}_j \quad 
    \mathrm{on}~~L^2\bigl( (0,\infty),dx \bigr) \quad (j=0,1,2,\cdots);\\
  & \widetilde{q}_j(x) 
    = \frac{(n-1)(n-3)}{4} \left( \frac{f'(x)}{f(x)} \right)^2 +
    \frac{n-1}{2}\frac{f''(x)}{f(x)} + \frac{\mu_j}{f^2(x)}.
\end{align*}
Since 
\begin{align*}
  \widetilde{q}_j(x) 
  = k(n-1) \frac{\cos 2r}{r} + O(r^{-2\alpha}),
\end{align*}
Lemma $2.1$ implies that $-\widetilde{L}_j$ has no eigenvalue on the interval $\left( 0 ,\infty \right) $ except for the special number $1$. 
Since the volume of $(M,g) = ({\bf R}^n,dr^2 + f^2(r) g_{S^{n-1}(1)})$ is infinite, zero is not eigenvalue. 
Thus we have proved Theorem $1.2$. 
%%
%%%%%%   Section 3     %%%%%%
%%
\section{Analytic propositions} 

In this section, we shall prepare some analytic propositions for the proof of Theorem $1.1$ and $1.3$. 

Let $(M,g)$ be an $n$-dimensional complete Riemannian manifold and $U$ an open subset of $M$. 
Assume that $E:= M - \overline{U}$ is an end with radial coordinates. 
We shall consider the eigenvalue equation 
\begin{align*}
    \Delta u + \lambda u = 0 \qquad {\rm on}~~E,
\end{align*}
where $\lambda >0$ is a constant.  

Let $\rho (r)$ be a $C^{\infty}$ function of $r \in [r_0,\infty)$, and put 
\begin{align*}
  v(x) = \exp \bigl( \rho (r(x)) \bigr) u(x) \qquad {\rm for}~~x \in E.
\end{align*}
Then it follows that $v$ satisfies on $B(r_0,\infty)$ the equation 
\begin{align*}
  & \Delta v - 2 \rho '(r) \frac{\partial v}{\partial r} + qv = 0,\\
  & q = |\nabla \rho (r)|^2 - \Delta \rho (r) + \lambda \nonumber \\
  & \hspace{1.8mm}  = \left| \rho '(r) \right|^2 - \rho ''(r) 
    - \rho '(r)\Delta r + \lambda.
\end{align*}
As is mentioned in section $1$, we denote by $dA$ the measures on each level surface $S(t)~~(t>0)$ induced from the Riemannian measure $dv_g$ on $(M,g)$. 
%%
%%%%%%%%%%%   Proposition 3.1   %%%%%%%%%%%
%%
\begin{prop}
For any $\psi \in C^{\infty}(M-\overline{U})$ and $r_0\le s< t$, we have
\begin{align*}
 & \int_{B(s,t)} \left\{ |\nabla v|^2 - q|v|^2 \right\}\psi \,dv_g \\
=& \left( \int_{S(t)} - \int_{S(s)} \right)
   \left( \frac{\partial v}{\partial r} \right) \psi v \,dA 
   - \int_{B(s,t)}
   \left\langle 
   \nabla \psi + 2 \psi \rho '(r) \nabla r , \nabla v 
   \right\rangle v \,dv_g.
\end{align*}
\end{prop}
\begin{proof}
Proposition $3.1$ is obtained by setting $c=0$ in [18], Proposition $3.1$. 
\end{proof}
%%
%%%%%%%%%%%   Proposition 3.2   %%%%%%%%%%%
%%
\begin{prop}
For any $r_0< s< t$ and $\gamma \in \mathbf{R}$, we have
\begin{align*}
 & \left( \int_{S(t)} - \int_{S(s)} \right) r^{\gamma}
   \left\{ \left(\frac{\partial v}{\partial r} \right)^2
   - \frac{1}{2}|\nabla v|^2 + \frac{1}{2}q|v|^2 \right\} \,dA \nonumber \\
= 
 & \int_{B(s,t)} r^{\gamma -1} \left\{ r(\nabla dr)(\nabla v,\nabla v)
   - \frac{1}{2} ( \gamma + r\Delta r ) \left( |\nabla v|^2 
   - \left( \frac{\partial v}{\partial r} \right)^2 \right) \right\} 
   \,dv_g  \nonumber \\
 & + \int_{B(s,t)}r^{\gamma -1} 
   \left\{ \frac{1}{2}( \gamma - r\Delta r ) + 2r\rho'(r) \right\}
   \left( \frac{\partial v}{\partial r} \right)^2 \,dv_g  \nonumber \\
 &  + \frac{1}{2} \int_{B(s,t)}r^{\gamma -1} 
   \left\{ 
   ( \gamma + r\Delta r ) q + r \frac{\partial q}{\partial r} 
   \right\}|v|^2 \,dv_g .
\end{align*}
\end{prop}
\begin{proof}
Proposition $3.2$ is obtained by setting $c=0$ in [18], Proposition $3.3$. 
\end{proof}
%%
%%%%%%%%%%%   Proposition 3.3   %%%%%%%%%%%
%%
\begin{prop}
Let $\nabla r, X_1, X_2, \cdots , X_{n-1}$ be an orthonormal base for the tangent space $T_xM$ at each point $x \in M - \overline{U}$ and $\psi_1(t)$ a function of $t\in [r_0,\infty)$. 
Then, for any real numbers $\gamma$, $\varepsilon$, and $0\le s<t$, we have
\begin{align*}
  & \int_{S(t)} r^{\gamma} \left\{ 
    \left(\frac{\partial v}{\partial r}\right)^2 + \frac{1}{2}q|v|^2 
    - \frac{1}{2}|\nabla v|^2 
    + \frac{ \gamma + \psi_1 (r) }{2r}
    \frac{\partial v}{\partial r}v \right\} \,dA \nonumber \\
  & + \int_{S(s)} r^{\gamma} \left\{ \frac{1}{2}|\nabla v|^2
    - \frac{1}{2}q|v|^2 - \left( \frac{\partial v}{\partial r} \right)^2
    - \frac{\gamma + \psi_1 (r) }{2r} \frac{\partial v}{\partial r}v \right\} 
    \,dA \nonumber \\
 =& \int_{B(s,t)} r^{\gamma -1} \left\{ r(\nabla dr)(\nabla v,\nabla v)
    - \frac{1}{2} \bigl( r\Delta r - \psi_1 (r) \bigr)
    \sum_{i=1}^{n-1} \left( dv(X_i) \right)^2 \right\} \,dv_g \\
  & + \int_{B(s,t)} r^{\gamma-1} \left\{ \gamma - 
    \frac{1}{2} \bigl( r\Delta r - \psi_1 (r) \bigr) + 2r \rho'(r) \right\}
    \left(\frac{\partial v}{\partial r}\right)^2 \,dv_g \nonumber \\
  & + \frac{1}{2} \int_{B(s,t)} r^{\gamma-1}
    \left\{ r\left( \frac{\partial q}{\partial r} \right)
    + q\bigl( r\Delta r - \psi_1 (r) \bigr) \right\} |v|^2 \,dv_g \nonumber \\
  & + \frac{1}{2} \int_{B(s,t)} r^{\gamma-1}
    \left\{ \left( \frac{\gamma-1}{r} + 2 \rho ' (r) \right) 
    \bigl( \gamma + \psi_1 (r) \bigr) + \psi_1 '(r) \right\} 
    \frac{\partial v}{\partial r}v \,dv_g.
\end{align*}
\end{prop}
\begin{proof}
Set $\psi = \frac{1}{2} r^{\gamma -1} \bigl( \gamma + \psi_1(r) \bigr) $ in Proposition $3.1$. 
Then we have
\begin{align*}
 & \frac{1}{2} \left( \int_{S(t)} - \int_{S(s)} \right) r^{\gamma-1} 
     \bigl( \gamma + \psi_1(r) \bigr) \frac{\partial v}{\partial r}v \,dA\\
= 
 & \frac{1}{2} \int_{B(s,t)} 
   r^{\gamma -1} \bigl( \gamma + \psi_1(r) \bigr) 
   \left\{ |\nabla v|^2 - q|v|^2 \right\} \,dv_g \\
 & \hspace{5mm} + \frac{1}{2} \int_{B(s,t)} r^{ \gamma - 2 } 
   \Bigl\{ \bigl( \gamma -1 + 2r \rho'(r) \bigr) 
   \bigl( \gamma + \psi_1(r) \bigr) 
   + r \psi_1 '(r) \Bigr\} \frac{\partial v}{\partial r} v \,dv_g.
\end{align*}
Addition of this equation to the equation in Proposition $3.2$, we get Proposition $3.3$. 
\end{proof}
%%
%%%%%%%%%%%   Lemma 3.1   %%%%%%%%%%%
%%
\begin{lem}
For any $ \beta \in {\bf R}$, we have 
\begin{align}
   \left( \int_{S(t)} - \int_{S(s)} \right) r^{\beta } |v|^2 \,dA
 = \int_{B(s,t)} r^{\beta }
   \left\{ \left( \Delta r + \frac{\beta }{r} \right)|v|^2
   + 2v\frac{\partial v}{\partial r} \right\} \,dv_g .
\end{align}
\end{lem}
\begin{proof}
Lemma $3.1$ follows from [18], Lemma $3.2$. 
\end{proof}
%%
%%%%%%%%%%%   Lemma 3.2   %%%%%%%%%%%
%%
\begin{lem}
For any $ m \in {\bf R}$, we have
\begin{align*}
 &  \int_{B(x,\infty)} r^{1-2m} 
    \left\{ |\nabla v|^2 - q|v|^2 \right\} \,dv_g  \nonumber \\
=&  - \frac{1}{2} \frac{d}{dx} \left( x^{1-2m} \int_{S(x)} |v|^2 \,dA \right)
    - \frac{1}{2} \int_{S(x)} r^{-2m} 
    \left\{ 2m - 1 - r \Delta r \right\} |v|^2 \,dA \nonumber \\
 &  - \int_{B(x,\infty)} r^{-2m} 
    \left( \frac{\partial v}{\partial r} \right) v \,dv_g,
\end{align*}
\end{lem}
\begin{proof}
Lemma $3.2$ is got by putting $c=0$ in [18], equations $(32)$ and $(33)$. 
\end{proof}
%%
%%%%%%   SECTION  4     %%%%%%
%%
\section{Faster than polynomial decay}

Let $(M,g)$ be an $n$-dimensional noncompact complete Riemannian manifold and $U$ an open subset of $M$. 
We assume that $E: = M - \overline{U}$ is an end with radial coordinates. 
We denote $r = {\rm dist}(U,*)$ on $E$. 
Let us set 
\begin{align*}
   \widetilde{g} = g - dr \otimes dr
\end{align*} 
and assume that there exists $r_0>0$ such that 
\begin{align}
  \left\{ \frac{f'(r)}{f(r)} - \frac{a}{r} \right\}  \,\widetilde{g} 
  \le \nabla dr 
  \le \left\{ \frac{f'(r)}{f(r)} + \frac{b}{r} \right\}  \,\widetilde{g} 
  \qquad {\rm on}~~B(r_0,\infty),  \tag{$*_1$} 
\end{align}
where $f(t)$ is a positive-valued $C^{\infty}$ function of $t\in [r_0,\infty)$, and $a$ and $b$ are positive constants. 
In the sequel, we shall often use the following notation for simplicity:
\begin{align*}
  & \widehat{a} = ( n - 1 ) a;~
    \widehat{b} = ( n - 1 ) b;~
    \widehat{b}_1 (r) = ( n - 1 ) b_1 (r);\\
  & A(r) = \frac{f'(r)}{f(r)};~K(r) = - \frac{f''(r)}{f(r)};\\
  & \widehat{A}(r) = (n-1) A(r);~\widehat{K}(r) = (n-1)K(r).
\end{align*} 
We assume that
\begin{align}
 & A(r) \ge \frac{A_0}{r} \qquad {\rm for}~~r \ge r_0 ; \tag{$*_2$} \\
 & 2(A_0-a) > (n-1)a + (n-1)b ; \tag{$*_3$} \\
 & \lim_{r\to \infty} A(r) = 0 ;  \tag{$*_4$} \\
 & - \widehat{K}_1 \le r \widehat{A}\,'(r) = - r \left( \widehat{A}(r)A(r) + \widehat{K} (r) \right) \le \widehat{K}_2 \qquad {\rm for}~~r \ge r_0 \tag{$*_5$}
\end{align}
and set
\begin{align}
  \widehat{K}_3=\max\{ \widehat{K}_1,\widehat{K}_2 \},\tag{$*_{5'}$}
\end{align}
where $A_0$, $K_1$, and $K_2$ are positive constants. 
%%
%%%%%%%%%%%   PROPOSITION 4.1   %%%%%%%%%%%
%%
\begin{prop}
Assume that there exist a positive-valued $C^{\infty}$ function $f(t)$ of $t \in [r_0,\infty)$ and constants $r_0 > 0$, $a$, and $b$ such that $(*_1)$, $(*_2)$, $(*_3)$, $(*_4)$ and $(*_5)$ hold. 
Let $\lambda > 0$ be a constant and $u$ a solution to 
\begin{align*}
     \Delta u + \lambda u = 0 \qquad {\rm on}~~B(r_0,\infty).
\end{align*}
Moreover, let 
\begin{align}
   \gamma > \frac{n-1}{2} ( a + b )  \tag{$*_9$}
\end{align} 
be a constant and assume that $u$ satisfies the condition$:$
\begin{align}
  \liminf_{t\to \infty} ~ t^{\gamma} \int_{S(t)}
  \left\{ \left( \frac{\partial u}{\partial r} \right)^2 + |u|^2 
  \right\} \,dA=0.
\end{align}
Moreover, we assume that 
\begin{align}
  \lambda > 
  \frac{(\widehat{K}_3)^2}{\left\{(-\varepsilon_0) -\widehat{a} \right\}\left\{2\gamma - \widehat{b}- (-\varepsilon_0) \right\}},
\end{align}
where $\varepsilon_0$ is a constant satisfying
\begin{align}
  \widehat{a} < (-\varepsilon_0) < 2\gamma - \widehat{b} \quad {\rm and} \quad 
  (-\varepsilon_0) < 2(A_0 - a) - \widehat{b}. 
\end{align}
Then, we have for any $m > 0$
\begin{align}
   \int_{B(r_0,\infty)} r^m \left\{ |u|^2 + |\nabla u|^2 \right\} 
    \,dv_g < \infty.
\end{align}
\end{prop}
\begin{rem}
{\rm The conditions $(21)$ and $(22)$ can be written as follows: 
\begin{align}
 \lambda > \frac{(\widehat{K}_3)^2}{\left\{ y_1 -\widehat{a} \right\}\left\{2\gamma - \widehat{b}- y_1 \right\}},
\end{align}
where $y_1= \min \left\{ \frac{2\gamma + \widehat{a} - \widehat{b}}{2}, 2(A_0-a) - \widehat{b} \right\}$.}
\end{rem}
\begin{proof}
We shall combine Proposition $3.3$ and Lemma $3.1$; put $ \psi_1 (r) = r \widehat{A}(r) + \varepsilon $ and $ \rho (r) = 0 $ in Proposition $3.3$; set $\beta = \gamma -1$ in Lemma $3.1$ and multiply $(19)$ by a positive constant $\alpha$. 
Then $v = u$ and $q = \lambda $ and 
\begin{align}
  & \int_{S(t)} r^{ \gamma }
    \left\{ \left( \frac{\partial u}{\partial r} \right)^2 
    + \frac{1}{2}\lambda |u|^2
    + \frac{ \gamma + r \widehat{A} (r)+ \varepsilon }{2r} 
    \frac{\partial u}{\partial r} u 
    + \frac{\alpha }{r}|u|^2 \right\} \,dA \\
  & + \int_{S(s)} r^{ \gamma }
    \left\{ \frac{1}{2}|\nabla u|^2 - \frac{1}{2}\lambda |u|^2 
    - \left( \frac{\partial u}{\partial r} \right)^2 
    - \frac{ \gamma + r \widehat{A} (r)+ \varepsilon }{2r}
    \frac{\partial u}{\partial r}u 
    - \frac{\alpha }{r}|u|^2 \right\} \,dA   \nonumber \\
\ge 
  & \int_{B(s,t)} r^{ \gamma - 1} \left\{ r(\nabla dr)(\nabla u,\nabla u)
    - \frac{1}{2} 
    \left( r\Delta r - r \widehat{A} (r) - \varepsilon \right)
    \widetilde{g} (\nabla u,\nabla u) \right\} \,dv_g \nonumber \\
  & + \int_{B(s,t)} r^{\gamma-1}
    \left\{ \gamma - \frac{1}{2} 
    \left( r\Delta r - r \widehat{A} (r) - \varepsilon \right) \right\}
    \left( \frac{\partial u}{\partial r} \right)^2 \,dv_g  \nonumber \\
  & + \int_{B(s,t)} r^{\gamma-1} 
    \left\{ \frac{\lambda }{2} 
    \left( r \Delta r - r \widehat{A} (r) - \varepsilon \right)
    + \alpha \left( \Delta r + \frac{\gamma -1}{r} \right) \right\} |u|^2 \,dv_g    \nonumber \\
  & + \int_{B(s,t)} r^{\gamma-1} 
    \left\{ r \widehat{A}\,'(r) + 2 \alpha + P(r) \right\}
    \frac{\partial u}{\partial r}u \,dv_g, \nonumber
\end{align}
where we set
\begin{align}
  P(r) = \left\{ \frac{(\gamma - 1)(n-1)}{2} 
  + 1 \right\} A(r) + \frac{( \gamma -1 )( \gamma + \varepsilon )}{2r} 
\end{align}
for simplicity. \\
Now, our assumptions $(*_9)$ and $(*_3)$ respectively imply that $ 2 \gamma - \widehat{b} > \widehat{a} $ and $ 2A_0 - 2a - \widehat{b} > \widehat{a} $. 
Hence, we can take $ - \varepsilon = -\varepsilon_0 $ so that
\begin{align}
  & 2 \gamma - \widehat{b} > (- \varepsilon_0) > \widehat{a};\\
  & 2 A_0 - 2a - \widehat{b} > (- \varepsilon_0).
\end{align}
Then, we see that 
\begin{align}
 & r(\nabla dr)(\nabla u,\nabla u) 
   - \frac{1}{2} 
   \left( r \Delta r - r \widehat{A} (r) - \varepsilon_0 \right)
   \widetilde{g} (\nabla u,\nabla u) \\
 & \hspace{20mm} \ge 
   \left\{ r A(r) - a 
   - \frac{r}{2} \left( \widehat{A}(r) + \frac{ \widehat{b} }{r} \right) 
   + \frac{r}{2} \widehat{A}(r) + \frac{ \varepsilon_0}{2} \right\} 
   \widetilde{g} (\nabla u,\nabla u) \nonumber \\ 
 & \hspace{20mm}\ge 
   \frac{1}{2} \{ 2 A_0 - 2a - \widehat{b} + \varepsilon_0 \} 
   \, \widetilde{g} (\nabla u,\nabla u), \nonumber 
\end{align}
where 
\begin{align}
  2 A_0 - 2a - \widehat{b} + \varepsilon_0 > 0 \qquad ({\rm by}~(28)).
\end{align}
Moreover, 
\begin{align}
 & \gamma 
   - \frac{1}{2} \left( r\Delta r - r \widehat{A} (r) - \varepsilon_0 \right)
   \ge \gamma 
   - \frac{r}{2} \left( \widehat{A}(r) + \frac{ \widehat{b} }{r} \right) 
   + \frac{r}{2} \widehat{A}(r) + \frac{\varepsilon_0}{2} \\
 & \hspace{39mm} 
   = \frac{1}{2}\{ 2 \gamma - \widehat{b} + \varepsilon_0 \} > 0 
   \qquad ({\rm by}~(27)) ;\nonumber 
\end{align}
and
\begin{align}
 & \frac{\lambda}{2} \left\{ r \Delta r - r \widehat{A} (r) - \varepsilon_0 
   + \alpha \left( \Delta r + \frac{\gamma -1}{r} \right) \right\} \\
\ge &
   \frac{\lambda}{2} \left\{ r \left( \widehat{A} (r) 
   - \frac{\widehat{a}}{r} \right) - r \widehat{A} (r) - \varepsilon_0 
   + \alpha \left( \widehat{A} (r) - \frac{\widehat{a}}{r} 
   + \frac{\gamma -1}{r} \right) \right\} \nonumber \\
 = &
   \frac{\lambda}{2} \left\{ - \widehat{a} - \varepsilon_0 
   + \alpha \left( \widehat{A} (r) - \frac{\widehat{a}}{r} 
   + \frac{\gamma -1}{r} \right) \right\} \nonumber \\
\ge & 
   \frac{\lambda}{2} \left\{ - \widehat{a} - \varepsilon_0 
   + \alpha \left( \widehat{A}_0 - \widehat{a} + \gamma -1 \right) \frac{1}{r}
   \right\} > 0 \qquad {\rm for}~r>>1 \quad ({\rm by}~(27)), \nonumber 
\end{align}
Now, from our assumption $(21)$, we take $ \alpha > 0 $ sufficiently small so that
\begin{align}
  ( - \widehat{a} - \varepsilon_0 ) ( 2\gamma - \widehat{b} + \varepsilon_0 ) 
  \lambda - ( \widehat{K}_3 + 2 \alpha )^2 > 0.
\end{align}
Then, a simple calculation using $(33)$ implies that there exists a constant $c_1=c_1(a,b,\gamma, K_3,\alpha) > 0$ such that 
\begin{align}
 & \frac{1}{2}\left( 2\gamma - \widehat{b} + \varepsilon_0 \right)
   \left(\frac{\partial u}{\partial r}\right)^2 
   + \frac{\lambda}{2}(-\widehat{a}-\varepsilon_0)|u|^2 
   + ( \widehat{K}_3 + 2\alpha)\frac{\partial u}{\partial r} u \\
 \ge 
 & 2c_1 \left\{ \left(\frac{\partial u}{\partial r}\right)^2 + |u|^2 \right\}.
   \nonumber 
\end{align}
Putting together $(*_5)$, $(*_{5'})$, $(25)$, $(26)$, $(29)$, $(30)$, $(31)$, $(32)$, and $(34)$, we see that there exist a constant $r_2=r_2(a,b,A_0,\gamma,\varepsilon_0,A(*)) > 0$ such that for any $ t > s \ge r_2$ the right hand side of $(25)$ is bounded from below by 
\begin{align*}
  c_1 \int_{B(s,t)} r^{\gamma-1} \bigl\{ |\nabla u|^2 + |u|^2 \bigr\} \,dv_g,
\end{align*}
that is, 
\begin{align}
 & \int_{S(t)} r^{ \gamma }
   \left\{ \left( \frac{\partial u}{\partial r} \right)^2 
   + \frac{1}{2}\lambda |u|^2
   + \frac{\gamma - \varepsilon }{2r} \frac{\partial u}{\partial r} u
   + \frac{\alpha }{r}|u|^2 
   \right\} \,dA \\
 & + \int_{S(s)} r^{ \gamma }
   \left\{ \frac{1}{2}|\nabla u|^2 - \frac{1}{2}\lambda |u|^2 
   - \left( \frac{\partial u}{\partial r} \right)^2 
   - \frac{\gamma - \varepsilon }{2r}\frac{\partial u}{\partial r}u 
   - \frac{\alpha }{r}|u|^2 
   \right\} \,dA   \nonumber \\
 \ge 
 & c_1 \int_{B(s,t)} r^{\gamma-1} 
    \bigl\{ |\nabla u|^2 + |u|^2 \bigr\} \,dv_g. \nonumber 
\end{align}
Besides, by Schwarz inequality, for $r \ge r_3 := \max\{ r_2, \frac{(\gamma - \varepsilon)^2}{4\alpha} \}$,
\begin{align*}
  - \left( \frac{\partial u}{\partial r} \right)^2 
  - \frac{\gamma - \varepsilon }{2r}\frac{\partial u}{\partial r}u 
  - \frac{\alpha }{r}|u|^2 
 \le 
  - \frac{1}{r} 
  \left\{ \alpha - \frac{ (\gamma - \varepsilon)^2 }{4r} \right\} |u|^2 
 \le 0,
\end{align*}
and moreover, $(20)$ implies that there exists a divergent sequence $\{t_i\}_{i=1}^{\infty}$ such that 
\begin{align*}
  \lim_{i \to \infty} \int_{S(t_i)} r^{ \gamma }
  \left\{ \left( \frac{\partial u}{\partial r} \right)^2 
  + \frac{1}{2}\lambda |u|^2
  + \frac{\gamma - \varepsilon }{2r} \frac{\partial u}{\partial r} u
  + \frac{\alpha }{r}|u|^2 
  \right\} \,dA = 0.
\end{align*}
Hence, substituting $t=t_i$ in $(35)$ and letting $ i \to \infty$, we get, for $s \ge r_3$, 
\begin{align}
  \frac{1}{2} \int_{S(s)} r^{\gamma}
  \left\{ |\nabla u|^2 - \lambda |u|^2 \right\} \,dA 
  \ge 
  c_1 \int_{B(s,\infty)} r^{\gamma-1} 
  \left\{ |\nabla u|^2 + |u|^2 \right\} \,dv_g.
\end{align}
Integrating this inequality with respect to $s$ over $[t,t_1]$ $(r_3 \le t < t_1)$, we have
\begin{align*}
 & 2 c_1 \int^{t_1}_t \,ds \int_{B(s,\infty)} r^{ \gamma - 1 }
   \left\{ |\nabla u|^2 + |u|^2 \right\} \,dv_g \\
 \le 
 & \int_{B(t,t_1)} r^{\gamma}
   \left\{ |\nabla u|^2 - q|u|^2 \right\} \,dv_g \\
 =
 & \left( \int_{S(t_1)} - \int_{S(t)} \right)
   r^{\gamma} \frac{\partial u}{\partial r}u \,dA
   - \gamma \int_{B(t,t_1)} r^{ \gamma - 1 }
   \frac{\partial u}{\partial r}u \,dv_g.
\end{align*}
In the last line, we have used the equation in Proposition $3.1$ with $\rho (r)=0$ and $ \psi = r^{\gamma} $. 
Since our assumption $(20)$ implies that 
\begin{align*}
  \liminf_{t_1\to \infty} \int_{S(t_1)} r^{\gamma}
  \frac{\partial u}{\partial r} u \,dA = 0,
\end{align*}
letting $t_1\to \infty$ and using Fubini's theorem, we have 
\begin{align}
 & 2 c_1 \int^{\infty}_t \,ds \int_{B(s,\infty)} r^{\gamma-1}
   \left\{ |\nabla u|^2 + |u|^2 \right\} \,dv_g \\
= 
 & 2 c_1 \int_{B(t,\infty)} (r-t) r^{\gamma-1}
   \left\{ |\nabla u|^2 + |u|^2 \right\} \,dv_g \nonumber \\
   \le 
 & \int_{S(t)} r^{\gamma}
   \left\{ \left( \frac{\partial u}{\partial r} \right)^2 + |u|^2 
   \right\} \,dA
   + \gamma \int_{B(t,\infty)} r^{\gamma-1}
   \left\{ \left( \frac{\partial u}{\partial r} \right)^2 + |u|^2 \right\}
   \,dv_g < \infty, \nonumber 
\end{align}
where the right hand side of this inequality is finite by $(36)$. 
Hence we see that the desired assertion $(23)$ holds for $m = \gamma$. 

Next, integrating this inequality $(37)$ with respect to $t$ over 
$[t_1,\infty)~(t_1\ge r_3)$ and using Fubini's theorem, we get
\begin{align*}
 & 2 c_1 \int_{B(t,\infty)} (r-t)^2 r^{\gamma - 1} 
   \left\{ |\nabla u|^2 + |u|^ 2 \right\} \,dv_g \\
 \le 
 & \int_{B(t,\infty)} r^{\gamma}
   \left\{ \left( \frac{\partial u}{\partial r} \right)^2 + |u|^2 \right\} 
   \,dv_g
   + \gamma \int_{B(t,\infty)} (r-t) r^{\gamma - 1}
   \left\{ \left( \frac{\partial u}{\partial r} \right) ^2 + |u|^2 \right\} 
   \,dv_g \\
 < 
 & \infty,
\end{align*}
where the right hand side of this inequality is finite by $(37)$. 
Thus, we see that the desired assertion $(23)$ holds for $ m = \gamma + 1 $. 
Repeating the integration with respect to $t$ shows that the assertion $(23)$ is valid for $m = \gamma + 2, \gamma + 3, \cdots $, therefore, for any $m > 0$. 
\end{proof}
%%
%%%%%%   SECTION  5     %%%%%%
%%
\section{Exponential decay}
%%
%%%%%%%%%    PROPOSITION 5.1       %%%%%%%%%
%%
\begin{prop}
To the assumptions in Proposition $4.1$, we shall add the following two assumptions$:$
\begin{align}
 & \frac{ \widehat{B}_0 }{ \sqrt{r} } 
   \ge \widehat{A}(r)  \qquad {\rm for}~r\ge r_0 ; \tag{$*_6$} \\
 & {\rm Ric}_g (\nabla r, \nabla r) \ge - \frac{ \widehat{b}_1 }{r}
   \qquad {\rm on}~~B(r_0,\infty), \tag{$*_7$} 
\end{align}
where $B_0$ and $ b_1 $ are positive constants. 
Then, we have 
\begin{align*}
  \int_{B(r_0,\infty)} e^{ \eta r } 
  \left\{ |u|^2 + |\nabla u|^2 \right\} \,dv_g < \infty \qquad 
  {\rm for~any}~0 < \eta < \eta_1 ( \lambda, a, b ) ,
\end{align*}
where we set
\begin{align*}
  & \eta_1 ( \lambda, a, b ) = 
 \begin{cases}
    \displaystyle \frac{ - c_6 + \sqrt{ ( c_6 )^2 + 4 c_7 c_0 \lambda } }{c_0} 
      & \text{if} \quad c_0 > 0, \vspace{1mm} \\
    \displaystyle \frac{ c_7 }{ c_6 } \lambda 
      & \text{if} \quad c_0 \le 0 ;
 \end{cases} \\
  & c_0 = 2 - A_0 + ( n + 1 )a + \widehat{b} ;~
    c_6 = \widehat{b}_1 + \frac{(\widehat{B}_0)^2}{8}; \\
  & c_7 = 2(A_0 -a) - \widehat{a} - \widehat{b} ~(>0).
\end{align*}
\end{prop}
\begin{proof}
In Proposition $3.3$, let 
\begin{align}
  & \rho (r) = m\log r ~~( m \ge  1, \widehat{b} ); ~~
    \gamma = 1; ~~\psi_1(r) = r \widehat{A}(r) + \varepsilon; \nonumber \\
  & \varepsilon = - 2 A_0 + 2a + \widehat{b} .
\end{align}
Then, 
\begin{align}
   v  
 & = r^m u; \nonumber \\
   q  
 & = \frac{m^2}{r^2} + \frac{m}{r^2} - \frac{m}{r}\Delta r + \lambda \\
   \ge 
 & \lambda 
   + \frac{m^2}{r^2} \left( 1 - \frac{ \widehat{b} - 1 }{m} \right) 
   - \frac{m}{r} \widehat{A}(r) ; \nonumber \\
   r \frac{\partial q}{\partial r} 
 & = - \frac{2m^2}{r^2} - \frac{2m}{r^2} + \frac{m}{r}\Delta r 
   - m \frac{\partial (\Delta r)}{\partial r}  ; \nonumber \\
   - \frac{\partial (\Delta r)}{\partial r}
 & = |\nabla dr|^2 + \mathrm{Ric}_g (\nabla r,\nabla r) \tag{$*_{10}$} \\
 & \ge 
   (n-1)\left\{ A(r) - \frac{a}{r} \right\}^2 
   - \frac{ \widehat{b}_1 }{r} \nonumber \\
 & = \widehat{A}(r) A(r) - \frac{ 2a \widehat{A}(r) + \widehat{b}_1 }{r} 
   + \frac{ \widehat{a} a }{r^2} \nonumber \\
 & \ge - \frac{ \widehat{b}_1 + P_1(r)}{r} , \nonumber
\end{align}
where we have used the identity $ - \frac{\partial (\Delta r)}{\partial r} = |\nabla dr|^2 + \mathrm{Ric}\,(\nabla r,\nabla r)$ (see \cite{K3}, Proposition $2.3$), and also, we set 
\begin{align*}
  P_1(r) = \frac{ a(2\widehat{A}_0-\widehat{a}) + \widehat{A}_0A_0 }{r}
\end{align*}
for simplicity. 
Note that $\lim_{r \to \infty} P_1(r) = 0$. \\
Moreover, by $(*_3)$,
\begin{align}
   r \Delta r - \psi_1(r) 
 & \ge r \widehat{A}(r) - \widehat{a} - r \widehat{A}(r) - \varepsilon 
   \nonumber \\
 & = - \widehat{a} - \varepsilon \\
 & = 2(A_0 -a) - \widehat{a} - \widehat{b} > 0 .\nonumber 
\end{align}
Hence, we have 
\begin{align*}
  r \frac{\partial q}{\partial r} 
  \ge 
  - \frac{2m^2}{r^2} \left( 1 + \frac{1}{m} \right) 
  - \frac{m}{r} \left( \widehat{b}_1 - \widehat{A}(r) 
  + \frac{\widehat{a}}{r} + P_1(r) \right) 
\end{align*}
and
\begin{align*}
  q \bigl( r \Delta r - \psi_1(r) \bigr) 
  \ge 
& \left\{ \lambda 
  + \frac{m^2}{r^2} \left( 1 - \frac{ \widehat{b} - 1 }{m} \right) 
  - \frac{m}{r} \widehat{A}(r)
  \right\} ( - \widehat{a} - \varepsilon ) \\
  = 
& ( - \widehat{a} - \varepsilon  ) \lambda 
  - \frac{m}{r} ( - \widehat{a} - \varepsilon ) \widehat{A}(r) 
  + \frac{m^2}{r^2} \left( 1 - \frac{ \widehat{b} - 1 }{m} \right) 
  ( - \widehat{a} - \varepsilon ) .
\end{align*}
Therefore,
\begin{align*}
  r \frac{\partial q}{\partial r} + q( r \Delta r - \psi_1(r) ) 
\ge 
  ( - \widehat{a} - \varepsilon ) \lambda 
  - \frac{m}{r} \left( \widehat{b}_1 + P_2 (r) \right) - \frac{m^2}{r^2} c_1(m),\end{align*}
where we set 
\begin{align}
 & P_2 (r) = ( - \widehat{a} - \varepsilon - 1 ) \widehat{A}(r) 
   + \frac{\widehat{a}}{r} + P_1(r); \nonumber \\
 & c_1(m) = 2 - ( - \widehat{a} - \varepsilon ) + \frac{1}{m} 
   \left\{ 2 + ( \widehat{b} - 1 )( - \widehat{a} - \varepsilon) \right\}.
\end{align}
Note that 
\begin{align*}
  - \widehat{a} - \varepsilon > 0~({\rm see}~(40));
  \quad \lim_{r \to \infty}  P_2 (r) = 0.
\end{align*}
We do not know the sign of the number $ c_0 = 2 - ( - \widehat{a} - \varepsilon ) = 2 - A_0 + ( n + 1 ) a + \widehat{b} $, and hence, for any $\theta \in (0,1)$, let us set
\begin{align}
  c_3 = \max \{ \theta, c_0 \} ~(>0). 
\end{align}
Then, 
\begin{align}
  c_1(m) \le c_3 + \theta \qquad {\rm for}~~m \ge m_1( A_0, a, b) ,
\end{align}
and hence, for $m \ge m_1$, 
\begin{align}
  r \frac{\partial q}{\partial r} + q( r \Delta r - \psi_1(r) ) 
 \ge 
  ( - \widehat{a} - \varepsilon ) \lambda 
  - \frac{m}{r} \left( \widehat{b}_1 + P_2 (r) \right) 
  - \frac{m^2}{r^2} ( c_3 + \theta ).
\end{align}
Besides, 
\begin{align}
 & r ( \nabla dr )(\nabla v, \nabla v) 
   - \frac{1}{2} \left( r\Delta r - r\widehat{A}(r) - \varepsilon \right) 
   \widetilde{g}(\nabla v, \nabla v) \\
   \ge 
 & \left\{ rA(r) - a 
   - \frac{1}{2} 
   \left( r\widehat{A}(r) + \widehat{b} - r\widehat{A}(r) 
   - \varepsilon \right) \right\}
   \widetilde{g}(\nabla v, \nabla v) \nonumber \\
   = 
 & \frac{1}{2} 
   \left\{ 2r A(r) - 2a - \widehat{b} + \varepsilon \right\}
   \widetilde{g}(\nabla v, \nabla v) \nonumber \\
   \ge
 & \frac{1}{2} 
   \left\{ 2A_0 - 2a - \widehat{b} + \varepsilon \right\}
   \widetilde{g}(\nabla v, \nabla v)
   = 0 \qquad ({\rm by}~(38)) \nonumber 
\end{align}
and 
\begin{align}
 & 1 - \frac{1}{2} 
   \left( r \Delta r - r \widehat{A}(r) - \varepsilon \right) + 2m \\
   \ge 
 & 1 - \frac{1}{2} 
   \left( 
   r \widehat{A}(r) + \widehat{b} - r \widehat{A}(r) - \varepsilon 
   \right)  + 2m  \nonumber \\
   = 
 & 2m + 1 - \frac{1}{2} ( \widehat{b} - \varepsilon) 
   = 2m + 1 - A_0 + a > 0 \nonumber \\
 & \hspace{30mm}  {\rm for}~~ 
   m > m_2 := 
   \max \left\{ m_1, \frac{ A_0 - a -1 }{2} \right\}. \nonumber 
\end{align}
Hence, by Proposition $3.3$, we have 
\begin{align}
 & \int_{S(t)} r \left\{ \left( \frac{\partial v}{\partial r} \right)^2 
   + \frac{1}{2} q|v|^2 
   + \frac{ 1 + \varepsilon + r \widehat{A}(r) }{2r} 
   \frac{\partial v}{\partial r}v \right\} \,dA \\
 & + \frac{1}{2} \int_{S(s)} r \bigl\{ |\nabla v|^2 - q |v|^2 \bigr\} \,dA 
   - \int_{S(s)} r \left\{ \left( \frac{\partial v}{\partial r} \right)^2 
   + \frac{ 1 + \varepsilon + r \widehat{A}(r) }{2r} 
   \frac{\partial v}{\partial r}v \right\} \,dA  \nonumber \\
 \ge \, 
 & \frac{1}{2} \int_{B(s,t)} \left\{ ( - \widehat{a} - \varepsilon ) \lambda 
   - \frac{m}{r} \left( \widehat{b}_1 + P_2 (r) \right) 
   - \frac{m^2}{r^2} ( c_3 + \theta ) \right\} 
   |v|^2 \,dv_g \nonumber \\
 & + ( 2m + 1 - A_0 + a ) \int_{B(s,t)}  
   \left( \frac{\partial v}{\partial r} \right)^2 \,dv_g \nonumber \\
 & + \int_{B(s,t)} 
   \left\{ \frac{m}{r} \left( 1 + \varepsilon + r \widehat{A}(r) \right) 
   + \frac{1}{2} \left( \widehat{A}(r) + r \widehat{A}\,'(r) \right) \right\}
   \frac{\partial v}{\partial r}v \,dv_g \nonumber 
\end{align}
for $m > m_2$. 
On the other hand, Lemma $3.1$ with $ \beta = 0 $ yields 
\begin{align}
   \left( \int_{S(t)} - \int_{S(s)} \right) |v|^2 \,dA
 = 
 & \int_{B(s,t)}  
   \left\{ \left( \Delta r \right) |v|^2
   + 2v\frac{\partial v}{\partial r} \right\} \,dv_g \\
 \ge 
 & \int_{B(s,t)}  
   \left\{ \left( \widehat{A}(r) - \frac{ \widehat{a} }{r} \right) |v|^2
   + 2v\frac{\partial v}{\partial r} \right\} \,dv_g \nonumber 
\end{align}
Multiplying this inequality $(48)$ by a constant 
\begin{align}
  \alpha 
  \in 
  \left( \frac{ ( \widehat{B}_0 )^2 }{16} + 1, 
  \frac{ ( \widehat{B}_0 )^2 }{16} + 2 \right)
\end{align}
and addition of it to $(47)$ make 
\begin{align}
 & \int_{S(t)} r  
   \left\{ \left( \frac{\partial v}{\partial r} \right)^2 
   + \frac{1}{2}q|v|^2 
   + \frac{ 1 + \varepsilon + r \widehat{A}(r) }{2r} 
   \frac{\partial v}{\partial r}v 
   + \frac{\alpha}{r}|v|^2 \right\} \,dA \\ 
 & + \frac{1}{2} \int_{S(s)} r \bigl\{ |\nabla v|^2 - q|v|^2 \bigr\} \,dA 
   \nonumber \\
 & - \int_{S(s)} r \left\{ \left( \frac{\partial v}{\partial r} \right)^2 
   + \frac{ 1 + \varepsilon + r \widehat{A}(r) }{2r} 
   \frac{\partial v}{\partial r}v 
   + \frac{ \alpha }{r}|v|^2 \right\} \,dA  \nonumber \\
 \ge \, 
 & \frac{1}{2} \int_{B(s,t)} 
   \left\{ ( - \widehat{a} - \varepsilon ) \lambda 
   + \alpha \widehat{A}(r) - \frac{ \alpha \widehat{a} }{r} 
   - \frac{m}{r} \left( \widehat{b}_1 + P_2 (r) \right) 
   - \frac{m^2}{r^2} ( c_3 + \theta ) \right\} 
   |v|^2 \,dv_g  \nonumber \\
 & + ( 2m + 1 - A_0 + a ) \int_{B(s,t)} 
   \left( \frac{\partial v}{\partial r} \right)^2 \,dv_g 
   + \int_{B(s,t)} P_3( \alpha, m, r ) 
   \frac{\partial v}{\partial r} v \,dv_g \nonumber 
\end{align}
for $m > m_2$, where we set 
\begin{align*}
  P_3( \alpha, m, r ) 
  = 2 \alpha + \frac{1}{2} \left( r \widehat{A}\,'(r) +  \widehat{A}(r) \right)
  + \frac{m}{r} \left( 1 + \varepsilon + r \widehat{A}(r) \right) 
\end{align*}
for simplicity. 
Substituting the inequality
\begin{align*}
  P_3( \alpha, m, r ) \frac{\partial v}{\partial r} v 
  \ge 
  - ( 2m + 1 - A_0 + a ) \left( \frac{\partial v}{\partial r} \right) ^2 
  - \frac{ P_3( \alpha, m, r )^2 }{ 4 ( 2m + 1 - A_0 + a ) } |v|^2 ,
\end{align*}
into $(50)$, we get 
\begin{align}
 & \int_{S(t)} r  
   \left\{ \left( \frac{\partial v}{\partial r} \right)^2 
   + \frac{1}{2}q|v|^2 
   + \frac{1}{2}\left( \frac{ 1 + \varepsilon }{r} +  \widehat{A}(r) \right) 
   \frac{\partial v}{\partial r}v + \frac{\alpha}{r}|v|^2 \right\} \,dA \\ 
 & + \frac{1}{2} \int_{S(s)} r \bigl\{ |\nabla v|^2 - q|v|^2 \bigr\} \,dA 
   \nonumber \\
 & - \int_{S(s)} r \left\{ \left( \frac{\partial v}{\partial r} \right)^2 
   + \frac{ 1 + \varepsilon + r \widehat{A}(r) }{2r} 
   \frac{\partial v}{\partial r}v 
   + \frac{ \alpha }{r}|v|^2 \right\} \,dA \nonumber \\
 \ge \, 
 & \int_{B(s,t)} H( \alpha, r, m ) \,|v|^2 \,dv_g , \nonumber 
\end{align}
where we set 
\begin{align*}
 & H( \alpha, r, m ) \nonumber \\
   = 
 & ( - \widehat{a} - \varepsilon ) \lambda + \alpha \widehat{A}(r) 
   - \frac{ \alpha \widehat{a} }{r}
   - \frac{ P_3( \alpha, m, r )^2 }{ 4 ( 2m + 1 - A_0 + a ) } 
   - \frac{m}{r} \left( \widehat{b}_1 + P_2(r) \right) 
   \nonumber \\
 & - \frac{m^2}{r^2} \left( c_3 + \theta \right) \nonumber \\
   = 
 & ( - \widehat{a} - \varepsilon ) \lambda 
   - \frac{ \left( 4 \alpha + r \widehat{A}\,'(r) +  \widehat{A}(r) \right)^2 }
   { 16 ( 2m + 1 - A_0 + a ) } + P_5(r,\alpha,m) \\
 & - \frac{m}{r} 
   \left\{ \widehat{b}_1 
   + \frac{ r ( \widehat{A}(r) )^2 }{ 4 ( 2 + \frac{1 - A_0 + a}{m} ) } 
   + P_6(r) 
   \right\} - \frac{m^2}{r^2} \left( c_3 + \theta \right) 
\end{align*}
and 
\begin{align*}
 & P_5(r,\alpha,m) = \alpha \widehat{A}(r) - \frac{ \alpha \widehat{a} }{r} 
   - \frac{ 4 \alpha + r \widehat{A}\,'(r) +  \widehat{A}(r) }
   { 4 \left( 2 + \frac{ 1 - A_0 + a }{m} \right) } 
   \left( \frac{ 1 + \varepsilon}{r} + \widehat{A}(r) \right) ; \\
 & P_6(r,m) = \frac{1}{ 4 ( 2 + \frac{1 - A_0 + a}{m})} 
   \left\{ 
   \frac{ (1 + \varepsilon)^2 }{r} + 2 (1 + \varepsilon) \widehat{A}(r)
   \right\} + P_2(r) .
\end{align*}
Here, note the following: $\lim_{r\to \infty}P_5(r,\alpha,m)=0$ uniformly for $\alpha \in \left( \frac{ ( \widehat{B}_0 )^2 }{16} + 1, \frac{ ( \widehat{B}_0 )^2 }{16} + 2 \right)$ and $m>m_2$ because $\bigl|r \widehat{A}\,'(r) \bigr| \le \widehat{K}_3$; $\lim_{ r \to \infty } P_6(r,m) = 0$ uniformly for $m > m_2$; $r( \widehat{A}(r) )^2 \le ( \widehat{B}_0 )^2$. 
Therefore, for any $\theta \in (0,1)$, if we take $r_4=r_4(A_0,a,A(*),b,K_3,\lambda)$ and $m_3 = m_3 (A_0,a,\widehat{B}_0)~( > m_2 )$ sufficiently large, we have 
\begin{align*}
 & H( \alpha, r, m ) 
   \ge 
   ( 1 - \theta ) ( - \widehat{a} - \varepsilon ) \lambda 
   - \frac{m}{r} c_4 - \frac{m^2}{r^2} \left( c_3 + \theta \right) \\
 & \hspace{30mm} {\rm for}~~\alpha \in 
  \left( \frac{ ( \widehat{B}_0 )^2 }{16} + 1, 
  \frac{ ( \widehat{B}_0 )^2 }{16} + 2 \right) ,~r \ge r_4,~
   {\rm and}~m \ge m_3,
\end{align*}
where we set 
\begin{align}
  c_4 =  \widehat{b}_1 + \frac{( 1 + \theta ) ( \widehat{B}_0 )^2}{8}.
\end{align}
Since $\lim_{r \to \infty}q = \lambda$ by $(39)$, Proposition $4.1$ implies that $|\nabla v|$ and $v$ are in $L^2\bigl( B(r_0,\infty),dv_g \bigr)$, and hence, 
\begin{align*}
  \liminf_{t\to \infty} \int_{S(t)} r  
  \left\{ \left( \frac{\partial v}{\partial r} \right)^2 
  + \frac{1}{2}q|v|^2 
  + \frac{1}{2}\left( \frac{ 1 + \varepsilon }{r} +  \widehat{A}(r) \right) 
  \frac{\partial v}{\partial r}v + \frac{\alpha}{r}|v|^2 
  \right\} \,dA = 0.
\end{align*}
Therefore, substituting appropriate divergent sequence $\{t_i\}$ for $t$ in $(51)$, and letting $t_i\to \infty$, we get
\begin{align}
 & \int_{S(s)} r \left\{ |\nabla v|^2 - q|v|^2 \right\} \,dA 
   - 2 \int_{S(s)} r \left\{ \left( \frac{\partial v}{\partial r} \right)^2 
   + \frac{ 1 + \varepsilon + r \widehat{A}(r) }{2r} 
   \frac{\partial v}{\partial r}v 
   + \frac{ \alpha }{r}|v|^2 \right\} \,dA \\
 \ge 
 & 2 \int_{B(s,\infty)} 
   \left\{ ( 1 - \theta ) ( - \widehat{a} - \varepsilon ) \lambda 
   - \frac{m}{r} c_4 - \frac{m^2}{r^2} \left( c_3 + \theta \right) 
   \right\}|v|^2 \,dv_g. \nonumber
\end{align}
Multiplying both side of $(53)$ by $ s^{ - 2m } $ and integrating it with respect to $s$ over $[x,\infty)$ $( x > r_4 )$, we have
\begin{align}
 & \int_{B(x,\infty)} r^{ 1 - 2m } 
   \left\{ |\nabla v|^2 - q|v|^2 \right\} \,dv_g \\
 & - 2 \int_{B(x,\infty)} r^{ 1 - 2m }
   \left\{ \left( \frac{\partial v}{\partial r} \right)^2 
   + \frac{ 1 + \varepsilon + r \widehat{A}(r) }{2r} 
   \frac{\partial v}{\partial r}v 
   + \frac{ \alpha }{r}|v|^2 \right\} \,dv_g \nonumber \\
 \ge 
 & 2 \int_{x}^{\infty} s^{ - 2m } \,ds 
   \int_{B(s,\infty)} 
   \left\{ ( 1 - \theta ) ( - \widehat{a} - \varepsilon ) \lambda 
   - \frac{m}{r} c_4 - \frac{m^2}{r^2} \left( c_3 + \theta \right) \right\}
   |v|^2 \,dv_g  \nonumber \\
 \ge 
 & 2 \int_{x}^{\infty}
   \left\{ ( 1 - \theta ) ( - \widehat{a} - \varepsilon ) \lambda 
    - \frac{m}{s} c_4 
    - \frac{m^2}{s^2} \left( c_3 + \theta \right) \right\} 
   s^{-2m} \,ds \int_{B(s,\infty)}|v|^2 \,dv_g \nonumber \\
 \ge 
 & 2 \left\{ ( 1 - \theta ) ( - \widehat{a} - \varepsilon ) \lambda 
    - \frac{m}{x} c_4 
    - \frac{m^2}{x^2} \left( c_3 + \theta \right) \right\} 
   \int_{x}^{\infty} s^{-2m} \,ds \int_{B(s,\infty)} |v|^2 \,dv_g \nonumber 
\end{align}
for $m \ge m_3$. 
Substitution of the equation in Lemma $3.2$ into $(54)$ makes 
\begin{align}
 & - \frac{1}{2} \frac{d}{dx} \left( x^{ 1 - 2m } \int_{S(x)}|v|^2 
   \,dA \right) 
   - \frac{1}{2} \int_{S(x)} r^{-2m} 
   \left\{ 2m - 1 \right\} |v|^2 \,dA \\
 & + \frac{1}{2} \int_{S(x)} r^{1 - 2m} ( \Delta r ) |v|^2 \,dA \nonumber \\
 & - \int_{B(x,\infty)} r^{1-2m} 
   \left\{ 2 \left( \frac{\partial v}{\partial r} \right)^2 
   + \left( \frac{ 2 + \varepsilon }{r} + \widehat{A}(r) \right) 
   \frac{\partial v}{\partial r}v 
   + 2 \frac{\alpha }{r}|v|^2 \right\} \,dv_g   \nonumber \\
 \ge 
 & 2 \left\{ ( 1 - \theta ) ( - \widehat{a} - \varepsilon ) \lambda 
   - \frac{m}{x} c_4 - \frac{m^2}{x^2}( c_3 + \theta ) \right\} 
   \int_{x}^{\infty} s^{-2m} \,ds \int_{B(s,\infty)} |v|^2 \,dv_g.\nonumber 
\end{align}
Here, by using $(*_6)$, the third term of the left hand side of $(55)$ is bounded as follows:
\begin{align*}
   \frac{1}{2} \int_{S(x)} r^{1 - 2m} ( \Delta r ) |v|^2 \,dA 
   \le 
 & \frac{1}{2} \int_{S(x)} r^{ \frac{1}{2} - 2m } 
   \left( \sqrt{r} \widehat{A}(r) + \frac{\widehat{b}}{\sqrt{r}} \right) 
   |v|^2 \,dA \\
   \le
 & \frac{ ( 1 + \theta) \widehat{B}_0 }{2\sqrt{x}} \int_{S(x)} r^{ 1 - 2m} 
   |v|^2 \,dA \quad {\rm for}~~
   x \ge r_5=\left( \frac{\widehat{b}}{\theta \widehat{B}_0} \right)^2.
\end{align*}
As for the fourth term of the left hand side of $(55)$, since $16 \alpha - ( \widehat{B}_0 )^2 > 1$ by $(49)$, we have
\begin{align*}
 & 2 \left( \frac{\partial v}{\partial r} \right)^2 + 
   \left( \frac{ 2 + \varepsilon }{r} + \widehat{A}(r) \right) 
   \frac{\partial v}{\partial r} v 
   + \frac{2\alpha}{r} |v|^2 \\ 
 \ge  
 & \frac{1}{2r} \left\{ \frac{ 2\alpha }{r} 
   - \frac{1}{8} \left( \frac{ 2 + \varepsilon }{r} 
   + \widehat{A}(r) \right)^2 \right\} |v|^2 \\
 = 
 & \frac{1}{8r^2} \left\{ 16 \alpha - r( \widehat{A}(r) )^2 + 
   2 ( 2 + \varepsilon )\widehat{A}(r) + \frac{( 2 + \varepsilon )^2}{r} 
   \right\} |v|^2 \\
 \ge 
 & \frac{1}{8r^2} \left\{ 16 \alpha - ( \widehat{B}_0 )^2 + 
   2 ( 2 + \varepsilon )\widehat{A}(r) + \frac{( 2 + \varepsilon )^2}{r} 
   \right\} |v|^2 \\ 
 \ge 
 & \frac{1}{8r^2} \left\{ 1 + 
   2 ( 2 + \varepsilon )\widehat{A}(r) + \frac{( 2 + \varepsilon )^2}{r} 
   \right\} |v|^2 \ge 0 \\
 & \hspace{60mm} {\rm for}~~r \ge r_6 = r_6( \varepsilon, \widehat{A}(*) )
   ~( \ge r_4).
\end{align*}
Therefore, for any $ m \ge m_4 = m_4( m_3, \theta ) $ and $ x \ge r_7 := \max\{ r_5, r_6 \} $, 
\begin{align}
 & - \frac{1}{2} \frac{d}{dx} \left (x^{1-2m} \int_{S(x)} |v|^2 \,dA \right)
   - \frac{ ( 1 - \theta  ) m }{x} \left( x^{ 1 - 2m } \int_{S(x)} |v|^2 
   \,dA \right) \\
 & + \frac{ ( 1 + \theta) \widehat{B}_0 }{2\sqrt{x}} 
   \left( x^{ 1 - 2m } \int_{S(x)} |v|^2 \,dA \right) \nonumber \\ 
 \ge 
 & 2 \left\{ ( 1 - \theta ) ( - \widehat{a} - \varepsilon ) \lambda 
   - \frac{m}{x} c_4 - \frac{m^2}{x^2}( c_3 + \theta ) \right\} 
   \int_{x}^{\infty} s^{-2m} \,ds \int_{B(s,\infty)} |v|^2 \,dv_g .\nonumber 
\end{align}
Now, let $c_5 >0$ be the solution of the quadratic equation $( 1 - \theta ) ( - \widehat{a} - \varepsilon ) \lambda - y c_4 - y^2( c_3 + \theta ) = 0$, 
that is, 
\begin{align}
  c_5 = \frac{ -c_4 + \sqrt{ c_4^2 + 4 (1-\theta)(-\widehat{a}-\varepsilon)(c_3+\theta)\lambda} }{ 2( c_3 + \theta ) }.
\end{align}
Moreover, for $ m \ge m_3 $ and $ x \ge r_7 $, we shall set 
\begin{align}
  \frac{m}{x} = c_5 >0
\end{align}
and 
\begin{align*} 
   F(x) = x^{1-2m} \int_{S(x)} |v|^2 \, dA = x \int_{S(x)} |u|^2 \,dA .
\end{align*}
Then, $(56)$ implies that 
\begin{align*} 
  - \frac{1}{2} F'(x) - ( 1 - \theta ) c_5 F(x) 
  + \frac{ ( 1 + \theta) \widehat{B}_0 }{2\sqrt{x}} F(x) \ge 0,
\end{align*}
and hence, 
\begin{align} 
  F'(x) \le - \left\{ 2( 1 - \theta ) c_5 
  - \frac{ ( 1 + \theta) \widehat{B}_0 }{2\sqrt{x}} \right\} F(x) 
  \le - 2( 1 - \theta )^2 c_5 F(x) 
\end{align}
for $x \ge r_8 = r_8( c_5, \theta, \widehat{B}_0 )~( \ge r_7 )$.
Thus, if we set $ G(x) = e^{ 2 ( 1 - \theta )^2 c_5 x} F(x) $, $(59)$ reduces to 
\begin{align*}
  G(x)' \le 0  \qquad {\rm for}~~x \ge r_8.
\end{align*}
Thus, 
$ G(x) \le G(r_8)$ for $x \ge r_8$, that is, 
\begin{align} 
  F(x) = x \int_{S(x)} |u|^2 \,dA 
  \le e^{ - 2( 1 - \theta )^2 c_5 x } G(r_8) \qquad {\rm for}~~x \ge r_8 .
\end{align}
On the other hand, in general, for $\beta > 0$ and $\delta > 0$,
\begin{align}
  \lim_{t \to +0} \frac{ \sqrt{ \beta^2 + 4\delta t} - \beta }{2t} 
  = \frac{\delta}{\beta}.
\end{align}
Now, in view of $(40)$, $(42)$, $(52)$, $(57)$, and $(61)$, we see that $(60)$ implies that 
\begin{align}
  \int_{B(r_0,\infty)} e^{ \eta r } |u|^2 \,dv_g < \infty 
  \qquad {\rm for~any}~~0 < \eta < \eta_1 (\lambda, a, b).
\end{align}

Next, we shall show that Proposition $3.1$ and $(62)$ yield 
\begin{align*}
  \int_{B(r_0,\infty)} e^{ \eta r } |\nabla u|^2 \,dv_g < \infty 
  \qquad {\rm for~any}~~0 < \eta < \eta_1 (\lambda, a, b).
\end{align*}
For that purpose, first consider the integral
\begin{align*}
  g(R)
  = 2 \int_{B(r_0,R)} e^{ \eta r } u \frac{\partial u}{\partial r} \,dv_g .
\end{align*} 
Then, Green's formula yields
\begin{align*}
   g(R)
= & \frac{1}{\eta} \int_{B(r_0,R)}
    \left\langle 
    \nabla \left( e^{ \eta r } \right),\nabla \left( u^2 \right) 
    \right\rangle \,dv_g \\
= & \frac{1}{\eta} \left( \int_{S(R)} - \int_{S(r_0)} \right)
    e^{ \eta r} |u|^2 \,dA 
    - \int_{B(r_0,R)} ( \Delta r + \eta ) e^{ \eta r }|u|^2 \,dv_g.
\end{align*}
Since $\lim_{r \to \infty } \Delta r = 0$, $(62)$ implies the existence of the limit, $\lim_{ R \to \infty }g(R)$. 
( Note that we do not assume $0 < \eta \le \eta_1 (\lambda, a, b)$ but assume $0 < \eta < \eta_1 (\lambda, a, b)$ ). 
In particular, 
\begin{equation}
   \liminf_{R\to \infty } e^{ \eta R}
   \left| \int_{S(R)} u \frac{\partial u}{\partial r} \,dA \right| = 0.
\end{equation}
In Proposition $3.1$, we put $\rho = 0$ and $\psi = e^{ \eta r}$. 
Then $v = u$, $ q = \lambda $, and 
\begin{align*}
  & \int_{B(r_0,R)} \left\{ |\nabla u|^2 - \lambda |u|^2 \right\} 
    e^{ \eta r} \,dv_g \\
= & \left( \int_{S(R)} - \int_{S(r_0)} \right)
    \frac{\partial u}{\partial r} u e^{ \eta r} \,dA
    -  \eta \int_{B(r_0,R)}
    e^{ \eta r} \frac{\partial u}{\partial r} u \,dv_g \\
\le 
  & \left( \int_{S(R)} - \int_{S(r_0)} \right)
    \frac{\partial u}{\partial r} u e^{ \eta r} \,dA
    + \frac{\eta}{2} \int_{B(r_0,R)}
    e^{ \eta r} \left\{ |\nabla u|^2 +  \eta ^2 |u|^2 \right\} \,dv_g.
\end{align*}
Hence,
\begin{align*}
  & \frac{1}{2} \int_{B(r_0,R)} e^{ \eta r} |\nabla u|^2 \,dv_g \\
 \le 
  & \left( \int_{S(R)} - \int_{S(r_0)} \right)
    \frac{\partial u}{\partial r} u e^{ \eta r} \,dA
    + \int_{B(r_0,R)}
    \left\{ \frac{ \eta ^2}{2} + \lambda \right\} e^{ \eta r} |u|^2 \,dv_g.
\end{align*}
Therefore, $(62)$ and $(63)$ imply that
\begin{equation}
  \int_{B(r_0,\infty)} e^{ \eta r} |\nabla u|^2 \,dv_g < \infty 
  \qquad {\rm for~any}~~0 < \eta < \eta_1 ( \lambda, a, b ) . 
\end{equation}
Thus, from $(62)$ and $(64)$, we get our desired result. 
\end{proof}
%%
%%%%%%   SECTION  6     %%%%%%
%%
\section{Vanishing on some neighborhood of infinity}
%%%
%%%%%%%%%     PROPOSITION  6.1     %%%%%%%
%%%
\begin{prop}
Under the assumptions in Proposition $5.1$, let us add the following assumption$:$
\begin{align}
  \lambda > 
  \left\{ \frac{ 4\widehat{b}_1 + (\widehat{B}_0)^2 }
  { 8\left( 2(A_0 - a) - \widehat{a} - \widehat{b} \right) } \right\}^2 
  \tag{$*_8$}. 
\end{align}
Then, we have $u\equiv 0$ on $B(r_0,\infty)$. 
\end{prop}
\begin{proof}
In proposition $3.3$, we shall put 
\begin{align*}
 & \gamma = 1 ;~
   \rho(r) = k r^{\theta }~
   \left( k\ge 1,~\theta \in \left(\frac{1}{2},1\right) \right) ;~ 
   \psi_1(r) = r \widehat{A}(r) + \varepsilon ; \\
 & \varepsilon = - 2 A_0 + 2a + \widehat{b} .
\end{align*}
Then, from
\begin{align}
 & - \frac{\partial (\Delta r)}{\partial r}
   = |\nabla dr|^2 + \mathrm{Ric}_g(\nabla r,\nabla r)
 \ge 
   - \frac{ \widehat{b}_1 + P_1(r) }{r} ; \tag{$*_{10}$} \\
 & r \widehat{A}(r) - \widehat{a} \le r \Delta r \le 
   r \widehat{A}(r) + \widehat{b} , \nonumber
\end{align}
we get
\begin{align}
 v = & e^{kr^{\theta }} u ;\\
 q = 
 & \lambda - \rho''(r) - \rho'(r) \Delta r + \left( \rho'(r) \right)^2 \\
 = 
 & \lambda + k \theta ( 1 - \theta ) r^{ \theta - 2 }
   - k \theta r^{\theta -1} \Delta r + k^2 \theta ^2 r^{ 2\theta - 2 } 
   \nonumber \\
 \ge 
 & \lambda 
   + k \theta \left( - \widehat{A}(r) + \frac{ 1 - \theta - \widehat{b} }{r} 
   \right) r^{ \theta - 1 } + k^2 \theta ^2 r^{ 2\theta - 2 } ; \nonumber \\
   r \frac{\partial q}{\partial r} 
 = 
 & - k \theta ( 1 - \theta )( 2 - \theta ) r^{\theta - 2} 
   + k \theta ( 1 - \theta ) r^{\theta -1} \Delta r  \nonumber \\
 & - k \theta r^{\theta } \frac{\partial (\Delta r)}{\partial r} 
   - 2 k^2 \theta ^2 ( 1 - \theta ) r^{ 2\theta - 2 }  \nonumber \\
 \ge 
 & - k \theta ( 1 - \theta )( 2 - \theta ) r^{ \theta - 2} 
   + k \theta ( 1 - \theta ) \left\{ r \widehat{A}(r) - \widehat{a} \right\} 
   r^{ \theta - 2}  \nonumber \\
 & - k \theta r^{ \theta - 1 } 
   \left\{ \widehat{b}_1 + P_1(r) \right\} 
   - 2 k^2 \theta ^2 ( 1 - \theta ) r^{ 2 \theta - 2 }  \nonumber \\
 = 
 & - k \theta \left\{ \widehat{b}_1 + P_2(r, \theta ) \right\} r^{ \theta - 1 }
   - 2 k^2 \theta ^2 ( 1 - \theta ) r^{ 2 \theta - 2 } , \nonumber 
\end{align}
where we set
\begin{align}
  P_2( r, \theta ) 
  = 
  ( 1 - \theta ) \left( - \widehat{A}(r) 
  + \frac{ \widehat{a} +  2 - \theta }{r} \right) + P_1(r)
\end{align}
for simplicity. 
Note that $\lim_{r \to \infty} P_2( r, \theta ) = 0$ uniformly with respect to $\theta \in (\frac{1}{2},1)$. 
In addition, 
\begin{align}
  r \Delta r - \psi_1(r) 
  \ge 
  - \widehat{a} - \varepsilon 
  = 2(A_0 - a) - \widehat{a} - \widehat{b} > 0 . \tag{$27$}
\end{align}
Therefore, on $B(r,\infty)$ $(r>>1)$, 
\begin{align}
 & r \frac{\partial q}{\partial r} + q ( r \Delta r - \psi_1(r) ) \\
 \ge 
 & - k \theta 
   \left\{ \widehat{b}_1 + P_2( r, \theta ) \right\} 
   r^{ \theta - 1 }
   - 2 k^2 \theta ^2 ( 1 - \theta ) r^{ 2 \theta - 2 } \nonumber \\
 & + \left\{ \lambda 
   + k \theta \left( - \widehat{A}(r) + \frac{ 1 - \theta - \widehat{b} }{r} 
   \right) r^{ \theta - 1 } + k^2 \theta ^2 r^{ 2\theta - 2 } \right\} 
   \left( - \widehat{a} - \varepsilon \right) \nonumber \\
=& \left( - \widehat{a} - \varepsilon \right) \lambda 
   - k \theta r^{ \theta - 1 } 
   \left\{ \widehat{b}_1 + P_2(r,\theta ) + \widehat{A}(r) 
   - \frac{ 1 - \theta - \widehat{b} }{r} \right\} \nonumber \\
 & + k^2 \theta ^2 r^{ 2\theta - 2 } 
   \left\{ \left( - \widehat{a} - \varepsilon \right) 
   - 2 ( 1 - \theta ) \right\} \nonumber \\
=& \left( - \widehat{a} - \varepsilon \right) \lambda 
   - k \theta r^{ \theta - 1 } 
   \left\{ \widehat{b}_1 + P_3( r, \theta ) \right\}
   + k^2 \theta ^2 r^{ 2\theta - 2 } c_8 ,  \nonumber
\end{align}
where we set 
\begin{align*}
   P_3( r, \theta ) 
 & = P_2( r, \theta ) + \widehat{A}(r) - \frac{1 - \theta - \widehat{b}}{r};\\
   c_8 
 & = \left( - \widehat{a} - \varepsilon \right) - 2 ( 1 - \theta ).
\end{align*}
Besides, 
\begin{align}
   1 - \frac{1}{2} \left( r \Delta r - \psi_1(r) \right) + 2 r \rho '(r) 
 \ge
 & 1 - \frac{1}{2} \left( r \widehat{A}(r) + \widehat{b} - r \widehat{A}(r) 
   - \varepsilon \right) + 2 k \theta r^{\theta} \\
 = 
 & 2 k \theta r^{\theta} + c_9 ; \nonumber
\end{align}
and 
\begin{align}
 & \frac{1}{2} \bigl\{ 2 \rho '(r) ( 1 + \psi_1(r) ) + \psi_1'(r) \bigr\}
   \frac{\partial v}{\partial r}v \\
  =
 & \left\{ k \theta r^{ \theta -1 } 
   \left( 1 + r \widehat{A}(r) + \varepsilon \right) 
   + \frac{1}{2} \left( \widehat{A}(r) + r \widehat{A}'(r) \right) \right\} 
   \frac{\partial v}{\partial r}v \nonumber \\
 \ge 
 & - \left\{ 2 k \theta r^{\theta} + c_9 \right\} 
   \left( \frac{\partial v}{\partial r} \right)^2 \nonumber \\
 & - \frac{ \left\{ k \theta r^{ \theta -1 } 
   \left( 1 + r \widehat{A}(r) + \varepsilon \right) 
   + \frac{1}{2} \left( \widehat{A}(r) + r \widehat{A}'(r) \right) \right\}^2 }
   { 4( 2 k \theta r^{\theta} + c_9 ) } |v|^2 , \nonumber 
\end{align}
where, we set 
\begin{align*}
   c_9 = 1 - \frac{1}{2} \left( \widehat{b} - \varepsilon \right)
\end{align*}
for simplicity. 
Here, by using $(*_4)$, $(*_5)$, and $(*_{5'})$, the coefficient of the term $|v|^2$ of $(70)$ is bounded as follows:
\begin{align}
 & \frac{ \left\{ k \theta r^{ \theta -1 } 
   \left( 1 + r \widehat{A}(r) + \varepsilon \right) 
   + \frac{1}{2} \left( \widehat{A}(r) + r \widehat{A}'(r) \right) \right\}^2 }
   { 4( 2 k \theta r^{\theta} + c_9 ) } \\
\le 
 & \frac{1}{8} k \theta r^{ \theta } \left( \widehat{A}(r) \right)^2 
  + k \theta r^{ \theta - 1 } P_4(r) + P_5(r) \nonumber \\
\le 
 & \frac{ ( \widehat{B}_0 )^2}{8} k \theta r^{ \theta -1 } 
  + k \theta r^{ \theta - 1 } P_4(r) + P_5(r) , \nonumber 
\end{align}
where $P_4(r)$ and $P_5(r)$ are functions of $r$ which are independent of $k\ge 1$ and $\theta \in ( \frac{1}{2},1 )$ and satisfy $ \lim_{r \to \infty}P_4(r) = \lim_{r \to \infty}P_5(r) = 0$. 
Since $ - \widehat{a} - \varepsilon > 0 $, we can choose $ \theta_0 \in ( \frac{1}{2},1 ) $ so that 
\begin{align}
  c_8 = \left( - \widehat{a} - \varepsilon \right) - 2 ( 1 - \theta ) 
  \ge ( 1 - \theta ) \left( - \widehat{a} - \varepsilon \right) > 0 
\end{align}
for any $\theta \in ( \theta_0, 1)$. 
Then, from $(68)$, $(69)$, $(70)$, and $(71)$, we see that 
\begin{align}
 & \frac{1}{2} \left\{ r \frac{\partial q}{\partial r} 
   + q ( r \Delta r - \psi_1(r) ) \right\} |v|^2
   + \frac{1}{2} \bigl\{ 2 \rho '(r) ( 1 + \psi_1(r) ) + \psi_1'(r) \bigr\} 
   \frac{\partial v}{\partial r} v \\
 & \hspace{30mm} + \left\{ 1 - \frac{1}{2} \left( r \Delta r - \psi_1(r) \right)   + 2 r \rho '(r) \right\} 
   \left( \frac{\partial v}{\partial r} \right)^2 \nonumber \\
 \ge 
 & \frac{1}{2} \left\{ \left( - \widehat{a} - \varepsilon \right) \lambda 
   - k \theta r^{ \theta - 1 } 
   \left( \widehat{b}_1 + \frac{ (\widehat{B}_0)^2 }{4} + P_6(r) \right) 
   + k^2 \theta ^2 r^{ 2\theta - 2 } c_8 \right\} |v|^2 , \nonumber 
\end{align}
where a function $P_6(r)$ is independent of $k \ge 1$ and $\theta \in (\theta_0,1)$ and satisfies $\lim_{r \to \infty}P_6(r) = 0$. 
Now, consider the discriminant $D$ of the quadratic equation
\begin{align*}
  \left( - \widehat{a} - \varepsilon \right) \lambda 
  - \left( \widehat{b}_1 + \frac{ (\widehat{B}_0)^2 }{4} \right)y 
  + \left( - \widehat{a} - \varepsilon \right) y^2 = 0 .
\end{align*}
Then, $D < 0$ by our assumption $(*_8)$, and hence, there exist constants $r_9(\ge r_0)$ and $\theta_1 \in (\theta_0, 1)$ such that 
\begin{align*}
  \left( - \widehat{a} - \varepsilon \right) \lambda 
  - \left( \widehat{b}_1 + \frac{ (\widehat{B}_0)^2 }{4} + P_6(r) \right) y 
  + ( 1 - \theta ) \left( - \widehat{a} - \varepsilon \right) y^2 > 0
\end{align*}
for $r \ge r_9$, $\theta \in [\theta_1,1)$, and $y \in {\bf R}$. 
Therefore, the right hand side of $(73)$ is nonnegative for any $k\ge 1$, $r\ge r_9$, and $\theta \in [\theta_1, 1)$. 
In the sequel, we shall fix $\theta \in [\theta_1, 1)$. 
Thus, we have for any $k\ge 1$ and $t > s \ge r_9$
\begin{align}
 & \int_{S(t)} r \left\{ 
   \left( \frac{\partial v}{\partial r} \right)^2 + \frac{1}{2}q|v|^2 
   - \frac{1}{2}|\nabla v|^2 
   + \left( \frac{ 1 + \varepsilon }{2r} + \widehat{A}(r) \right) 
   \frac{\partial v}{\partial r}v \right\} \,dA \\
 & + \int_{S(s)} r \left\{ \frac{1}{2}|\nabla v|^2
   - \frac{1}{2}q|v|^2 - \left(\frac{\partial v}{\partial r}\right)^2
   - \left( \frac{ 1 + \varepsilon }{2r} + \widehat{A}(r) \right) 
   \frac{\partial v}{\partial r}v \right\} \,dA \ge 0. \nonumber 
\end{align}
In view of $(65)$, we see that Proposition $5.1$ implies that 
\begin{align*}
  \liminf_{t\to \infty} \int_{S(t)} r 
  \left\{ |\nabla v|^2 + |v|^2 \right\} \,dA = 0.
\end{align*}
Hence, substituting an appropriate divergent sequence $\{t_i\}$ for $t$ in $(74)$, and letting $t_i\to \infty$, we see that 
\begin{align}
  \int_{S(s)} r \left\{ \frac{1}{2}|\nabla v|^2
  - \frac{1}{2}q|v|^2 - \left(\frac{\partial v}{\partial r}\right)^2
  - \left( \frac{ 1 + \varepsilon }{2r} + \widehat{A}(r) \right) 
  \frac{\partial v}{\partial r}v \right\} \,dA \ge 0
\end{align}
for all $ s \ge r_7$ and $ k \ge 1$. 
On account of the facts
\begin{align*}
   \left( \frac{\partial v}{\partial r} \right)^2
   = 
 & \left\{ k^2 \theta^2 r^{ 2 \theta - 2 } |u|^2
   + 2k \theta r^{\theta -1} \frac{\partial u}{\partial r} u
   + \left( \frac{\partial u}{\partial r} \right)^2 \right\} 
   e^{2kr^{\theta}} , \\
   |\nabla v|^2
   = 
 & \left\{ k^2 \theta^2 r^{2\theta -2} |u|^2
   + 2k \theta r^{\theta -1} \frac{\partial u}{\partial r}u
   + |\nabla u|^2 \right\} e^{2kr^{\theta}},
\end{align*}
and $(66)$, the left hand side of $(75)$ is written as follows:
\begin{align*}
  \left\{ k^2 I_1(s) + k I_2(s) + I_3(s) \right\} e^{2kr^{\theta}},
\end{align*}
where 
\begin{align*}
   I_1(s) = - \theta^2 s^{ 2 \theta - 2 } \int_{S(s)} |u|^2 \,dA;
\end{align*}
$I_2(s)$ and $I_3(s)$ is independent of $k$. 
Thus, for any fixed $ s\ge r_9$, the inequality $ k^2 I_1(s) + k I_2(s) + I_3(s) \ge 0 $ holds for all $ k \ge 1$. 
Therefore, $ I_1(s) = 0 $ for any fixed $s\ge r_9$, that is, $u \equiv 0$ on $B( r_9, \infty)$. 
The unique continuation theorem implies that $ u \equiv 0$ on $E = M - \overline{U}$. 
\end{proof}

\noindent {\it Proof of Theorem $1.3$}\\
Proposition $6.1$ proves the first part of Theorem $1.3$. 
The second part of Theorem $1.3$ is proved as follows: 
assume that $2 > \widehat{a} + \widehat{b}$. 
If $\lambda > \lambda_1( 1, a, b, A_0, B_0, K_3, b_1 )$ is an eigenvalue of $ - \Delta $ and $u$ is an corresponding non-trivial eigenfunction, then $u,\nabla u \in L^2(M,dv_g)$, in particular, $u,\nabla u \in L^2(E,dv_g)$. 
However, the first part of Theorem $1.3$, which is just proved above, implies that 
\begin{align*}
   \liminf_{t \to \infty}~t \int_{S(t)} 
   \left\{ \left( \frac{\partial u}{\partial r} \right)^2 + |u|^2 \right\}
   \,dA \neq 0 .
\end{align*}
Therefore, there exist positive constants $c_{10}$ and $r_9$ such that 
\begin{align*}
   t \int_{S(t)} 
   \left\{ \left( \frac{\partial u}{\partial r} \right)^2 + |u|^2 \right\}
   \,dA > c_{10} \qquad {\rm for}~t \ge r_9.
\end{align*}
Hence, dividing the both sides of this inequality by $t$ and integrating it with respect to $t$ over $[r_9,\infty)$, we get
\begin{align*}
   \int_{B(r_9,\infty)} 
   \left\{ \left( \frac{\partial u}{\partial r} \right)^2 + |u|^2 \right\}
   \,dA > \int_{r_9}^{\infty} \frac{c_{10}}{t} = + \infty.
\end{align*}
This contradicts the fact that $u,\nabla u \in L^2(E,dv_g)$. 
Thus, we have proved the second part of Theorem $1.3$. 
\vspace{2mm}

\noindent {\it Proof of Theorem $1.1$}\\
Theorem $1.1$ follows from Theorem $1.3$ and the comparison theorem in Riemannian geometry. 
Kasue \cite{Kasue} is a good reference for the comparison theorem in Riemannian geometry. 
%%
%%%%%%   SECTION  7     %%%%%%
%%
\section{Further discussion}
Applying the comparison theorem in Riemannian geometry, we obtain the following corollary from Theorem $1.3$:
%%%%%%%%%%%   Corollary 7.1   %%%%%%%%%%%
%%
\begin{cor}
Let $(M,g)$ be an $n$-dimensional noncompact complete Riemannian manifold and $E$ is an end of $M$ with radial coordinates. 
We denote $r = {\rm dist}(\partial E,*)$ on $E$. 
Assume that there exist positive constants $r_0,a,b$,and $A_0$ such that 
$$
  \frac{A_0}{r} \le A(r) \le  \frac{B(r)}{\sqrt{r}} \quad {\rm for}~ r\ge r_0;~2(A_0-a) > \widehat{a} + \widehat{b},
$$
where $B:[r_0,\infty) \to (0,\infty)$ is a positive-valued function satisfying $\lim_{t\to \infty}B(t) = 0$. 
Assume also that $|K(r)|=o(r^{-1})$ and 
\begin{align*}
   \left\{ A(r_0) - \frac{a}{r_0} \right\} \, \widetilde{g} 
   \le 
   \nabla & dr 
   \le 
   \left\{ A(r_0) + \frac{b}{r_0} \right\} \, \widetilde{g} \quad {on}~S(r_0);\\
   - \frac{b}{r} \left\{ 2A(r) - \frac{1-b}{r} \right\}
   \le 
   K_{{\rm rad.}} & - K(r) 
   \le \frac{a}{r}\left\{ 2A(r) - \frac{a+1}{r} \right\}
   \quad {on}~B(r_0,\infty).
\end{align*}
Then, $\sigma_{{\rm pp}}(-\Delta) = \emptyset$. 
\end{cor}

Substituting several functions for $f(r)$ in Theorem $1.3$, we get many examples of manifold with no eigenvalue. For example, letting $\frac{1}{2} < \alpha <1$ and $p>0$ be constants and substituting $f(r) = \exp \left\{ p(1-\alpha)r^{1-\alpha} \right\}$ in Theorem $1.3$, we get the following:
%%
%%%%%%%%%%%   Corollary 7.2   %%%%%%%%%%%
%%
\begin{cor}
Let $(M,g)$ be an $n$-dimensional noncompact complete Riemannian manifold and $E$ is an end of $M$ with radial coordinates. 
We denote $r = {\rm dist}(\partial E,*)$ on $E$. 
Assume that there exist positive constants $r_0$, $p$, $\frac{1}{2} < \alpha <1$, and $b>1$ such that the following hold$:$
\begin{align*}
   \frac{p}{(r_0)^{\alpha}} \, \widetilde{g} 
  & \le 
   \nabla dr 
   \le 
   \left\{ \frac{p}{(r_0)^{\alpha}} + \frac{b}{r_0} \right\}\, \widetilde{g} \qquad {on}~S(r_0);\\
   -\frac{p(2b-\alpha)}{r^{\alpha +1}} 
  & \le 
   K_{{\rm rad.}} + \frac{p^2}{r^{2\alpha}} 
   \le \frac{p\alpha}{r^{\alpha +1}}
   \qquad {on}~B(r_0,\infty).
\end{align*}
Then, $\sigma_{{\rm pp}}(-\Delta)=\emptyset$. 
\end{cor}

%%
%%%%%%    Reference    %%%%%%
%%

\vspace{3mm}

\end{document}